\documentclass[12pt,reqno]{amsart}

\usepackage{amssymb,hyperref}
\usepackage{amsmath}
\usepackage{etoolbox}

\newtheorem{theorem}{Theorem}[section]
\newtheorem{proposition}[theorem]{Proposition}
\newtheorem{lemma}[theorem]{Lemma}
\newtheorem{corollary}[theorem]{Corollary}

\theoremstyle{definition}
\newtheorem{definition}[theorem]{Definition}
\newtheorem{remark}[theorem]{Remark}

\usepackage[usenames,dvipsnames,svgnames,table]{xcolor}
\usepackage{tikz}
\usepackage{tikz-cd}

\newcommand{\desu} [2] []
{\dfrac {\partial #1} {\partial #2}}

\newcommand{\lra}{\longrightarrow}

\newcommand{\Thetatwo}{{\mathbf L}}
\newcommand{\diffmap}{\phi}
\newcommand{\lb}{K_S(2\Delta)}
\newcommand{\tphi}{\widetilde\phi}
\newcommand{\one}{\alpha}
\newcommand{\due}{\beta}
\newcommand{\three}{\gamma}
\newcommand{\isom}{\Psi}
\newcommand{\thetacar}{F}
\newcommand{\tras}{\mathcal{T}}
\renewcommand{\Im}{\operatorname{Im}}

\numberwithin{equation}{section}

\begin{document}

\baselineskip=15pt

\title[Theta functions and projective structures]{Theta functions and projective structures}

\author[I. Biswas]{Indranil Biswas}

\address{Department of Mathematics, Shiv Nadar University, NH91, Tehsil
Dadri, Greater Noida, Uttar Pradesh 201314, India}

\email{indranil.biswas@snu.edu.in, indranil29@gmail.com}

\author[A. Ghigi]{Alessandro Ghigi}
	
\address{Dipartimento di Matematica, Universit\`a di Pavia, via
Ferrata 5, I-27100 Pavia, Italy}
	
\email{alessandro.ghigi@unipv.it}
	
\author[L. Vai]{Luca Vai}
	
\address{Dipartimento di Matematica, Universit\`a di Pavia, via
Ferrata 5, I-27100 Pavia, Italy}
	
\email{luca.vai01@universitadipavia.it}



\date{}

\begin{abstract} 
Given a compact Riemann surface $X$, we consider the line, in the space of
holomorphic sections of $2\Theta$ on $J^0(X)$, orthogonal to all the
sections that vanish at the origin. This line produces a natural
meromorphic bidifferential on $X\times X$ with a pole of order two
on the diagonal. This bidifferential is extensively investigated. In
particular we show that it produces a projective structure on $X$
which is different from the standard ones.
\end{abstract}

\maketitle

\tableofcontents

\section{Introduction}\label{se1}

Let $X$ be a compact connected Riemann surface, and let
$\Delta\, \subset\, S\,:=\,X\times X$ be the reduced diagonal divisor.
Holomorphic and meromorphic sections of $K_S \,=\, \Omega^2_S$ provide a
tool to study holomorphic 1-forms and Hodge theory on $X$. Especially
useful are the so-called \emph{bidifferentials}, i.e., meromorphic
sections of $K_S$ whose polar divisor contains no subset of the form
$\{x\}\times X$ or $X\times \{x\}$. Many bidifferentials have in fact
polar divisor equal to $2\Delta$, hence can be seen as holomorphic
sections of $K_S(2\Delta)$. Recently, sections of $K_S(2\Delta)$ have been
used to describe projective structures on $X$, see \cite {BR1,
BR,Tyu}. Several bidifferentials were known
classically \cite{Fay,Gu2,Tyu}, but most of them are not intrinsic, in
the sense that they depend on some choice: a symplectic basis of
$H_1(X,\,\mathbb Z)$ or a theta characteristic. Other bidifferentials
are intrinsic but are not defined for every $X$ (this is the case for
Klein's bidifferential $\omega_X$; see \cite[p. 22]{Fay} and \cite[p. 185]{Tyu}).

One classically known bidifferential, that we will denote by $\eta_X$, is important in the
differential-geometric study of the Torelli map $j\,:\, M_g \,\longrightarrow\, A_g$. This map is an
embedding outside the hyperelliptic locus $HE_g$. So $J\,:=\,j (M_g - HE_g)$ is
a submanifold of $A_g$ in the orbifold sense. As $A_g$ is provided
with the Siegel metric, it is natural to consider the extrinsic
differential geometry of this submanifold $J \,\subset\, A_g$. After
appropriate identifications, the second fundamental form of this
submanifold coincides with the multiplication by $\eta_X$; see
\cite{EPG} and \cite{cpt}. See also \cite{Gu2} and \cite{Lo} for
related results. Although the section $\eta_X$ is constructed using
just basic Hodge theory, its construction is not completely
straightforward. One starts working on $X \times \{x\}$ for any
$x\,\in\, X$, then one glues everything together to a section on $X\times X$.

 Anyway $\eta_X$ turns out to be completely intrinsic in the sense
 that its construction does not require making any choice. Moreover, by
 \cite{BR}, the bidifferential $\eta_X$ produces a projective structure
 on $X$, which is also completely intrinsic. This projective structure has been
 studied in \cite{bcfp}, where it is shown that it is actually different from
 the projective structure obtained by uniformization.

 The main goal of the present paper is to introduce another completely
 intrinsic bidifferential, that we will denote by $\sigma_X$. Its
 construction is cleaner than that of $\eta_X$, in the sense that one
 works from the beginning on the Jacobian of $X$ and on $X\times X$.
 Nevertheless the two constructions have a similar flavor, and at the
 beginning we expected that $\eta_X \,=\, \sigma_X$ for any $X$. Our main
 result is that in fact they differ for general $X$, at least for
 genus $1$.

 The section $\sigma_X$ is constructed using theta functions. There is
 a canonical theta divisor on $J^{g-1}(X)$, where
 $g\,=\, \text{genus}(X)$. Using this divisor it can be shown that
 there is a canonical line bundle $2\Theta$ on the Jacobian $J^0(X)$.
 The vector space $H^0(J^0(X),\, 2\Theta)$ admits a Hermitian inner product
 which is intrinsic up to a constant scalar multiple. Since $2\Theta$ is base
 point-free, $H^0(J^0(X),\, 2\Theta)$ contains a unique line
 $\mathbb S$ which is orthogonal to all sections vanishing at the origin
 ${\mathcal O}_X\, \in\, J^0(X)$. This is the content of Section
 \ref{section:one}.

 Next let $\phi\,:\, S\,=\,X\times X \,\lra\, J^0(X)$ be the
 difference map $\phi(x,y) \,:=\, \mathcal O_{X}(x-y)$. Using the
 formalism of determinant bundles we show in Proposition \ref {prop1}
 that the pullback $\phi^*(2\Theta)$ is isomorphic to
 $K_{S }(2\Delta)$; see Remark \ref{pr-re} for an elegant proof of the proposition
communicated by the referee. (Another proof is described later in the paper.)
 The bidifferential $\sigma_X$ is by definition the unique element of
 the line $\phi^*\mathbb S$ with biresidue 1 along the diagonal
 (the line $\mathbb S$ as defined above). We close Section \ref{section:two}
 recalling several results from a paper \cite{vgvg} by van Geemen and
 van der Geer which describe both the kernel and the image of the map
\begin{gather}
\phi^*\,\,:\,\, H^0(J^0(X),\,2\Theta) \,\,\lra\,\, H^0(S ,\,
\Omega_{S}^2 (2\Delta)).
\label{pbintro}
\end{gather}

The next step (comprising Sections \ref {sec:pullback:special} and
\ref{section:5}) is a rather explicit study of the map $\phi^*$ using
the classical formalism of theta functions, as in e.g. \cite{Fay,Gu3}.
We first concentrate on a special class of sections $s_\zeta$ of
$ H^0(J^0(X),\,2\Theta) $. We are able to compute the divisor of $\phi^*s_\zeta$. Among other
things, this yields another proof of Proposition \ref{prop1} which
says that $\phi^*(2\Theta) $ is isomorphic to $K_{S}(2\Delta)$ (see
Proposition \ref{prop:omega:zeta}).

Next we compute $\phi^*(s)$ for arbitrary sections
$s \in H^0(J^0(X),\,2\Theta)$. To pass from the analysis of
$\phi^*(s_\zeta)$ to the general case of $\phi^*(s)$ we rely on a deep
result: Fay's Trisecant Identity. We start by differentiating twice
this identity. After some nontrivial computations we arrive at a very
explicit formula for the pullback map \eqref{pbintro} (see Proposition
\ref{prop:operativepullback}). This formula is inspired by a result in
\cite{Gu3}. It is crucial for the rest of the paper. This is the
content of Section \ref{section:5}.

The final Section \ref {section:6} compares the above mentioned
sections $\sigma_X$ and $\eta_X$. We start by recalling briefly the
construction of $\eta_X$ given in \cite{EPG}. It should be mentioned
that this bidifferential had been studied earlier, at least by Gunning
\cite{Gu2}. In fact, the characterization of $\eta_X$ given by Gunning
is very useful for the problem at hand. After some work we deduce from
it a system of equations (real analytic in the period matrix) which
describe the locus where the sections $\eta_X$ and $\sigma_X$
coincide. It is shown that these two sections are different in general
(Theorem \ref{final}). This is proved by considering the special case
of $g\,=\,1$. It is natural to expect that the two projective
structures differ also for the generic point in $M_g$ for any
$g$. Nevertheless it seems hard to generalize our proof to higher
genera, on the one hand because the computations with theta functions
become difficult, on the other hand because the generic period matrix
does not correspond to a Jacobian as soon as we have $g \geq 4$.

Since $\sigma_X$ is an intrinsic bidifferential with biresidue 1, it
gives rise to a new intrinsic projective structure on $X$. Our last
result shows that in general this structure is different from the one
produced by $\eta_X$. Indeed we show that in the special case of
$g\,=\,1$ the two projective structures are different for a generic
point $[X]\in M_{1,1}$ (see Corollary \ref{cor-d}). Corollary
\ref{cor-d} is actually proved using Theorem \ref{final}.

Intrinsic projective structures on compact Riemann surfaces fit nicely
in the picture put forward in \cite{BFPT,BGT} (using the result of
\cite{fpt}) and in fact they are the motivation for that picture. As
shown in those papers there are bijections between three objects:
\begin{enumerate}
\item  smooth global sections of the
bundle of projective structures over the moduli space of curves,

\item connections on the dual of the Hodge line bundle over the moduli
space of curves, and

\item closed $(1,1)$-forms on the moduli space of curves representing a particular
Dolbeault cohomology class.
\end{enumerate}
(See the Introduction to \cite{BGT} and Theorem 3.5 in the same
paper.) The projective structure induced by $\sigma_X$ gives rise to
a new element of (1). It would be very interesting to identify the corresponding
elements of the sets (2) and (3). In other words, it would be
interesting to characterize the projective structure obtained from
$\sigma_X$ in terms of its $\bar{\partial}$, as is done in \cite
{BFPT} for the projective structure induced by the uniformization and
in \cite {BGT} for the projective structure induced by $\eta_X$. We
hope to come back to these questions in the future.

\section{A line of sections}\label{section:one}

Let $X$ be a compact connected Riemann surface. The holomorphic
cotangent bundle of $X$ will be denoted by $K_X$. For any
$d\, \in\, {\mathbb Z}$, let $J^d(X)$ denote the component of the
Picard group of $X$ that parametrizes the isomorphism classes of all
holomorphic line bundles on $X$ of degree $d$. Let $g$ be the genus of
$X$. Then $J^{g-1}(X)$ has a natural reduced effective divisor
\begin{equation}\label{e1}
\Theta\, :=\, \{L\, \in\, J^{g-1}(X)\, \big\vert\,\, H^0(X,\, L)\, \not=\, 0\}
\, \subset\, J^{g-1}(X)\, ,
\end{equation}
which is known as the \textit{theta divisor}. For any
$L\, \in\, J^{g-1}(X)$ the Riemann--Roch theorem says that
$\dim H^0(X,\, L)-\dim H^1(X,\, L)\,=\, 0$, and hence any
$L\, \in\, J^{g-1}(X)$ lies in $\Theta$ if and only if
$H^1(X,\, L)\, \not=\, 0$.

Take a theta characteristic $\thetacar$ on $X$. So $\thetacar$ is a
holomorphic line bundle on $X$ of degree $g-1$ such that
$\thetacar\otimes\thetacar\,=\, K_X$. Consider the holomorphic
isomorphism
\begin{equation}\label{e2}
\tras_{\thetacar}\, :\, J^{0}(X)\, \longrightarrow\, J^{g-1}(X)\, ,\ \
L\, \longmapsto\, L\otimes\thetacar\, .
\end{equation}
Let
\begin{equation}\label{e3}
\Theta_{\thetacar}\, :=\, \tras^{-1}_{\thetacar}(\Theta) \, \subset\, J^{0}(X)
\end{equation}
be the inverse image of the divisor $\Theta$ defined in
\eqref{e1}. Define the involution
\begin{equation}\label{io}
\iota\, :\, J^{0}(X)\, \longrightarrow\, J^{0}(X)\, ,\ \ L\, \longmapsto\, L^*\, .
\end{equation}

The following lemma is well-known. Its proof below was supplied by a referee.

\begin{lemma}\label{lem1}\mbox{}
\begin{enumerate}
\item The divisor $\Theta_{\thetacar}$ in \eqref{e3} satisfies the
equality
$$
\iota (\Theta_{\thetacar})\,=\, \Theta_{\thetacar}.
$$

\item The isomorphism class of the holomorphic line bundle on $J^0(X)$
$$
{\mathcal O}_{J^{0}(X)}(\Theta_{\thetacar} +
\iota(\Theta_{\thetacar})) \,=\, {\mathcal O}_{J^{0}(X)}(2\Theta_{\thetacar})
$$
is actually independent of the choice of the theta characteristic $\thetacar$.
\end{enumerate}
\end{lemma}

\begin{proof}
Statement (1) follows immediately from the Serre duality and the fact that
$\thetacar\otimes\thetacar\,=\, K_X$. (A more general result
can be found in \cite[Theorem 11.2.4]{lange-birkenhake}.)

Take another theta characteristic ${\thetacar}'$ on $X$. So $\xi\, :=\, {\thetacar}'\otimes {\thetacar}^*$
is a holomorphic line bundle on $X$ of order two. The second statement follows from the
Theorem of the Square, \cite{Mu}, which says that
$$
{\mathcal O}_{J^{0}(X)}(\Theta_{{\thetacar}'})\otimes {\mathcal O}_{J^{0}(X)}(\Theta_{{\thetacar}'})
\, =\, {\mathcal O}_{J^{0}(X)}(\mathcal{T}^*_{\xi^2} \Theta_{\thetacar})\otimes
{\mathcal O}_{J^{0}(X)}(\Theta_{\thetacar})\,=\, {\mathcal O}_{J^{0}(X)}(2\Theta_{\thetacar}),
$$
where $\mathcal{T}_{\xi^2} \, :\, J^{0}(X)\,\longrightarrow\, J^{0}(X)$ is the translation by $\xi^{\otimes 2}$
(so it is the identity map of $J^0(X)$).
\end{proof}

The holomorphic line bundle
${\mathcal O}_{J^{0}(X)}(\Theta_{\thetacar} +
\iota(\Theta_{\thetacar})) \,=\, {\mathcal
O}_{J^{0}(X)}(2\Theta_{\thetacar})$ on $J^{0}(X)$ will be denoted by
$\mathbf L$.

The complex vector space
\begin{equation}\label{e7}
\mathbb{V}\, \,=\,\, H^0(J^{0}(X),\, {\mathbf L})
\end{equation}
is of dimension $2^g$, because $c_1(\Theta_F)$ is a principal
polarization. It has a Hermitian structure which is unique up to
multiplication by a constant scalar; in other words, $\mathbb{V}$ has
a canonical conformal class of Hermitian structures. We will briefly
recall a construction of this conformal class of Hermitian structures.

The cup product
$H^1(X,\, {\mathbb Z})\otimes H^1(X,\, {\mathbb Z})\,
\longrightarrow\, H^2(X,\, {\mathbb Z})\,=\, \mathbb Z$ on $X$
produces a K\"ahler form
\begin{equation}\label{kf}
\omega_0\, \in\, \Omega^{1,1}(J^{0}(X))
\end{equation}
on $J^{0}(X)$. We have $c_1({\mathbf L})\,=\, 2[\omega_0]$. Fix a
Hermitian structure $h$ on ${\mathbf L}$ such that the curvature of
the Chern connection on
\begin{equation}\label{kf2}
({\mathbf L},\, h)
\end{equation}
coincides with $2\omega_0$. Any two such Hermitian structures on
$\mathbf L$ differ by multiplication by a constant scalar. Then we
have the $L^2$ metric on the vector space $\mathbb{V}$ (in \eqref{e7})
defined by
\begin{equation}\label{kf3}
\langle s,\, t\rangle\,=\, \int_{J^{0}(X)} \langle s,\, t\rangle_h\cdot \omega^{g}_0
\end{equation}
for $s,\, t\, \in\, \mathbb{V}$. Since any two choices of $h$ differ
by multiplication by a constant scalar, the conformal class of the
$L^2$ metric on $\mathbb{V}$ in \eqref{kf3} does not depend on the
choice of $h$.

Next we note that the complete linear system $\big\vert{\mathbf L}\big\vert$ is
base-point free, meaning for every point $z\, \in\, J^{0}(X)$, there
is a section $s_z\, \in\,\mathbb{V}$ such that $s_z(z)\, \not=\, 0$
\cite[p. 60, Application 1(iii)]{Mu}. Let
\begin{equation}\label{v0}
V_0\, \subset\, \mathbb{V}
\end{equation}
be the subspace of codimension one defined by all sections that vanish
at the identity element $0\,=\, [{\mathcal O}_X]\, \in\, J^{0}(X)$. We
have its orthogonal complement
\begin{equation}\label{e8}
{\mathbb S}\,\, :=\,\, V^\perp _0\,\, \subset\,\, \mathbb{V}\, .
\end{equation}
Note that $\mathbb S$ is identified with the fiber ${\mathbf L}_0$, of
$\mathbf L$ over $0$, by the evaluation map that sends any
$s\,\in\, {\mathbb S}\, \subset\, H^0(J^{0}(X),\, {\mathbf L})$ to its
evaluation $s(0)$ at $0$.

\section{The section $\sigma_X$}\label{section:two}

As done in the introduction, set
\begin{gather*}
S\,\,:=\,\,X\times X.
\end{gather*}
For $i\,=\,1,\, 2$, let
\begin{equation}\label{ep}
p_i\, :\, X\times X\, \longrightarrow\, X
\end{equation}
be the natural projection to the $i$-th factor. Let
$$
\Delta\, :=\, \{(x,\, x)\, \in\, S\, \mid\,\, x\, \in\, X\}\,
\subset\, S,
$$
be the reduced diagonal divisor.

Let
\begin{equation}\label{e9}
\phi\, :\, X\times X\, \longrightarrow\, J^{0}(X)
\end{equation}
be the holomorphic map defined by
$(x,\, y)\, \longmapsto\, {\mathcal O}_X(x-y)$.

\begin{proposition}\label{prop1}
The pullback $\phi^*{\mathbf L}$ of the line bundle in \eqref{e7} by
$\phi$ in \eqref{e9} satisfies the following:
$$
\phi^*{\mathbf L}\,=\, (p^*_1 K_X)\otimes (p^*_2 K_X)\otimes{\mathcal
O}_{S}(2\Delta)\, .
$$
\end{proposition}

\begin{proof}
Take any Poincar\'e line bundle ${\mathbb L}\, \longrightarrow\, X\times 
J^{g-1}(X)$. Let $\Phi\, :\, X\times J^{g-1}(X)\, \longrightarrow\, J^{g-1}(X)$
be the natural projection. Consider the line bundle
$$
\mathbb{D}(\mathbb{L})\ :=\ (\det \Phi_*\mathbb{L}) \otimes (\det R^1 \Phi_*\mathbb{L})^*
\ \longrightarrow\ J^{g-1}(X)
$$
(see \cite[Ch.~V,\S~6]{Ko} for the determinant bundle). This line bundle
$\mathbb{D}(\mathbb{L})$ does not depend on the choice of the Poincar\'e line bundle
$\mathbb{L}$. The fiber of $\mathbb{D}(\mathbb{L})$ over any $L\, \in\, J^{g-1}(X)$
is identified with
$$
d(L)\, :=\, \bigwedge\nolimits^{\rm top} H^0(X,\, L)\otimes
\bigwedge\nolimits^{\rm top} H^1(X,\, L)^*.
$$

The holomorphic line bundle
${\mathcal O}_{J^{g-1}(X)}(\Theta) \, \longrightarrow\, J^{g-1}(X)$,
where $\Theta$ is the divisor in \eqref{e1}, 
is canonically identified with the line bundle $\mathbb{D}(\mathbb{L})$
(see \cite[p.~567]{NR}, \cite[p.~396]{Fal}). Consequently,
the fiber of ${\mathcal O}_{J^{g-1}(X)}(\Theta)$ over any line bundle
$\zeta\, \in\, J^{g-1}(X)$ is canonically identified with the line
$$
\theta_\zeta\, :=\, \mathbb{D}(\mathbb{L})^*_\zeta \,=\, d(\zeta)^*.
$$

Note that since $\chi(\zeta)\, =\, 0$, the automorphisms of $\zeta$
given by the nonzero scalar multiplications actually act trivially on
the line $\theta_\zeta$. Fix a theta characteristic $\tau$ on
$X$. From the above description of the fibers of
${\mathcal O}_{J^{g-1}(X)}(\Theta)$ it follows that the fiber of the
line bundle ${\mathcal O}_{J^0(X)}(\Theta_{\tau})$ (see \eqref{e3})
over any $L\, \in\, J^0(X)$ is the line
$$
\theta_{L\otimes{\tau}}\,=\, d(L\otimes{\tau})^*\, .
$$

Take any point $(x,\, y)\, \in\, X\times X$, and consider the line bundle
$\mathbf L$ in \eqref{e7}. From the above
description of the fibers of ${\mathcal O}_{J^0(X)}(\Theta_{\tau})$ we
conclude that the fiber of the line bundle $\phi^*\mathbf L$ (see
\eqref{e9}) over $(x,\, y)$ is the line
\begin{equation}\label{el}
d({\mathcal O}_X(x-y)\otimes{\tau})^*
\otimes d({\mathcal O}_X(y-x)\otimes{\tau})^*\, .
\end{equation}

To compute $d({\mathcal O}_X(x-y)\otimes{\tau})$, first consider the
short exact sequence of sheaves on $X$
$$
0\, \longrightarrow\, {\mathcal O}_X(x-y)\otimes{\tau}\,
\longrightarrow\, {\mathcal O}_X(x)\otimes{\tau} \longrightarrow\,
({\mathcal O}_X(x)\otimes{\tau})_y \, \longrightarrow\, 0\, ,
$$
where the sheaf $({\mathcal O}_X(x)\otimes{\tau})_y$ is supported on
$y$. From this short exact sequence we conclude that
\begin{equation}\label{e10}
d({\mathcal O}_X(x)\otimes{\tau})\,=\, d({\mathcal O}_X(x-y)\otimes{\tau})
\otimes ({\mathcal O}_X(x)\otimes{\tau})_y\, .
\end{equation}

Next, using the short exact sequence of sheaves on $X$
$$
0\, \longrightarrow\, {\tau}\, \longrightarrow\, {\mathcal
O}_X(x)\otimes{\tau} \longrightarrow\, ({\mathcal
O}_X(x)\otimes{\tau})_x \, \longrightarrow\, 0\, ,
$$
where the sheaf $({\mathcal O}_X(x)\otimes{\tau})_x$ is supported on
$x$, we have
\begin{equation}\label{e10b}
d({\mathcal O}_X(x)\otimes{\tau})\,=\, d(\tau)\otimes ({\mathcal O}_X(x)\otimes{\tau})_x
\,=\, d(\tau)\otimes{\tau}_x\otimes {\mathcal O}_X(x)_x\, .
\end{equation}
Using the Poincar\'e adjunction formula (see \cite[p.~146]{GH}),
\begin{equation}\label{e10a} {\mathcal O}_X(x)_x\,=\, T_xX\, .
\end{equation}
Hence combining \eqref{e10b} and \eqref{e10a},
$$d({\mathcal O}_X(x)\otimes{\tau})\,=\, d(\tau)\otimes {\tau}_x\otimes
T_xX\,=\, d(\tau)\otimes {\tau}^*_x\, ,$$ because
${\tau}^{\otimes 2}\,=\, (TX)^*$ (recall that $\tau$ is a theta
characteristic). Combining this with \eqref{e10} it follows that
\begin{equation}\label{e11}
d({\mathcal O}_X(x-y)\otimes{\tau})\,=\, 
d(\tau)\otimes {\tau}^*_x\otimes {\tau}^*_y\otimes
{\mathcal O}_X(-x)_y\, .
\end{equation}

Interchanging $x$ and $y$, from \eqref{e11} we have
\begin{equation}\label{e12}
d({\mathcal O}_X(y-x)\otimes{\tau})\,=\,
d(\tau)\otimes {\tau}^*_y\otimes {\tau}^*_x\otimes
{\mathcal O}_X(-y)_x\, .
\end{equation}
Note that both ${\mathcal O}_X(-x)_y$ and ${\mathcal O}_X(-y)_x$ are
identified with ${\mathcal O}_{X\times
X}(-\Delta)_{(x,y)}$. Therefore, substituting \eqref{e11} and
\eqref{e12} in \eqref{el}, we have
$$
d({\mathcal O}_X(x-y)\otimes{\tau})^* \otimes d({\mathcal
O}_X(y-x)\otimes{\tau})^* \,=\, \left((p^*_1 K_X)\otimes (p^*_2
K_X)\otimes{\mathcal O}_{X\times X}(2\Delta)\right)_{(x,y)}\, .
$$

Here we use that $d(\tau)$ is a line independent of $(x,y)$.

Note that $d({\mathcal O}_X(x-y)\otimes{\tau})^* \otimes d({\mathcal
O}_X(y-x)\otimes{\tau})^*$ is the fiber of $\phi^*\mathbf L$ over
$(x,\, y)$, and hence the proposition follows.
\end{proof}

\begin{remark}\label{pr-re}
A referee supplied the following elegant proof of Proposition \ref{prop1}.
Fix a theta characteristic $\tau$ on $X$ with $H^0(X,\, \tau)\,=\, 0$.
Then for any $x,\, y\, \in\, X$, on we have
$$
\phi^{-1}(\iota^*\Theta_\tau) \cap (
X\times\{y\})\ = \ \{z\, \in\, X\,\,\big\vert\,\,
H^0(X,\, \tau\otimes {\mathcal O}_X(y-z))\, \not=\, 0\},
$$
$$
\phi^{-1}(\iota^*\Theta_\tau) \cap (\{x\}\times X)\ = \ \{z\, \in\, X\,\,\big\vert\,\,
H^0(X,\, \tau\otimes {\mathcal O}_X(z-x))\, \not=\, 0\}.
$$
Since $h^0(X,\, \tau)\,=\, 0$ and $h^0(X, \tau \otimes {\mathcal O}_X (y) )
 =1 $,  it can be deduced that
$$
\phi^* {\mathcal O}_{J^0(X)}(\iota^*\Theta_\tau)\big\vert_{X\times \{y\}}
\ = \ \tau\otimes {\mathcal O}_X(y),\ \ \,
\phi^* {\mathcal O}_{J^0(X)}(\iota^*\Theta_\tau)\big\vert_{\{x\}\times X}
\ = \ \tau\otimes {\mathcal O}_X(x).
$$
On the other hand,
$$
\phi^* {\mathcal O}_{J^0(X)}(\Theta_\tau)\big\vert_{X\times \{y\}}
\ = \ \tau\otimes {\mathcal O}_X(y),\ \ \,
\phi^* {\mathcal O}_{J^0(X)}(\Theta_\tau)\big\vert_{\{x\}\times X}
\ = \ \tau\otimes {\mathcal O}_X(x).
$$

Consequently, we have
$$
\phi^*{\mathbf L}\big\vert_{X\times \{y\}}\,=\, K_X\otimes {\mathcal O}_X(2y)\,=\,
((p^*_1 K_X)\otimes (p^*_2 K_X)\otimes{\mathcal
O}_{S}(2\Delta))\big\vert_{X\times \{y\}},
$$
$$
\phi^*{\mathbf L}\big\vert_{\{x\}\times X}\,=\, K_X\otimes {\mathcal O}_X(2x)\,=\,
((p^*_1 K_X)\otimes (p^*_2 K_X)\otimes{\mathcal
O}_{S}(2\Delta))\big\vert_{\{x\}\times X}.
$$
Now Proposition \ref{prop1} follows from the see--saw theorem \cite{Mu}.
\end{remark}

\begin{remark}\label{rem1}
The Poincar\'e adjunction formula says that the restriction of
${\mathcal O}_{S}(\Delta)$ to the divisor $\Delta\, \subset\,S$
coincides with the normal bundle of $\Delta$, which, in turn, is
identified with $TX$ after we identify $X$ with $\Delta$ via the map
$x\, \longmapsto\, (x,\, x)$. Consequently, the restriction of the
line bundle
$(p^*_1 K_X)\otimes (p^*_2 K_X)\otimes{\mathcal O}_{S}(2\Delta)$ to
$\Delta$ coincides with the trivial line bundle
${\mathcal O}_\Delta$ on $\Delta$.
\end{remark}

The map $\phi$ in \eqref{e9} satisfies the condition
$\phi(\Delta)\,=\, 0\, \in\, J^0(X)$, and hence for any nonzero
section $s\, \in\, \mathbb S$ (see \eqref{e8}), the evaluation
$\phi^*s(x,\, x)\, \in\, (\phi^*{\mathbf L})_{(x,x)}$ is nonzero for
every $x\, \in\, X$. Consequently, using Proposition \ref{prop1} it is
deduced that
\begin{equation}\label{e13}
\phi^*{\mathbb S} \, \subset\, H^0(X\times X,\, (p^*_1 K_X)\otimes
(p^*_2 K_X)\otimes{\mathcal O}_{S}(2\Delta))
\end{equation}
is a $1$-dimensional subspace (in particular, it is not zero
dimensional), and furthermore, for any nonzero
$$v \,\in\,\phi^*{\mathbb S} \, \subset\, H^0(X\times X, \, (p^*_1 K_X)\otimes (p^*_2
K_X)\otimes{\mathcal O}_{S}(2\Delta))$$ the evaluation
$v(x,\, x)\, \in\, ((p^*_1 K_X)\otimes (p^*_2 K_X)\otimes{\mathcal
O}_{S}(2\Delta))_{(x,x)}$ is nonzero for every $x\, \in\, X$.

Note that the subspace $\phi^*{\mathbb S}$ in \eqref{e13} does not
depend on the choice of the isomorphism (in Proposition \ref{prop1})
between
$(p^*_1 K_X)\otimes (p^*_2 K_X)\otimes{\mathcal O}_{S}(2\Delta)$ and
$\phi^*{\mathbf L}$.

Now in view of Remark \ref{rem1} we have the following:

\begin{proposition}\label{prop2}
There is a unique section
$$\sigma_X\, \in\, \phi^*{\mathbb S} \, \subset\, H^0(S,\,
(p^*_1 K_X)\otimes (p^*_2 K_X)\otimes{\mathcal O}_{S}(2\Delta))$$
such that the restriction of $\sigma_X$ to $\Delta$ coincides with
the constant function $1$ on $\Delta$ using the identification
between
$((p^*_1 K_X)\otimes (p^*_2 K_X)\otimes{\mathcal
O}_{S}(2\Delta))\big\vert_\Delta$ and ${\mathcal O}_\Delta$.
\end{proposition}

Let
\begin{equation}\label{e14}
f_X\, :\, X\times X\, \longrightarrow\, X\times X\, ,\ \ (x,\, y) \,\longmapsto\,
(y,\, x)
\end{equation}
be the involution of $X\times X$. We note that
$$
f^*_X ((p^*_1 K_X)\otimes (p^*_2 K_X)\otimes{\mathcal O}_{S}(2\Delta))
\,=\, (p^*_1 K_X)\otimes (p^*_2 K_X)\otimes{\mathcal O}_{S}(2\Delta)\,
.
$$

\begin{lemma}\label{lem2}
The section $\sigma_X$ in Proposition \ref{prop2} satisfies the
equation
$$
f^*_X \sigma_X \,=\, \sigma_X.
$$
\end{lemma}

\begin{proof}
 The map $\iota$ in \eqref{io} preserves the divisor
 $\Theta_{\tau} +\iota(\Theta_{\tau})$ (see Lemma \ref{lem1}(1))
 and $\iota(0) = 0$.
 Therefore the
involution $\iota$ admits a unique lift to a holomorphic involution $\iota'$ of the line bundle
${\mathbf L}\, :=\, {\mathcal O}_{J^{0}(X)}(\Theta_{\tau} +
\iota(\Theta_{\tau}))$ on $J^{0}(X)$
such that $\iota' $ acts as the identity on $\mathbf L_0$.
% 
% ; this involution of
% ${\mathbf L}$ over the involution $\iota$ will be denoted by
% $\iota'$.
This involution $\iota'$ of ${\mathbf L}$ produces an
involution
\begin{gather*}
\iota''\, :\, {\mathbb V}\,:=\, H^0(J^{0}(X),\, {\mathbf L})\,
\longrightarrow\, {\mathbb V}
\label{iopp}
\end{gather*}
defined by $s\, \longmapsto\, \iota' \circ s \circ \iota$.

The K\"ahler form $\omega_0$ in \eqref{kf} is preserved by $\iota$,
and hence the hermitian structure $h$ on ${\mathbf L}$ in \eqref{kf2}
is preserved by $\iota'$ up to a constant scalar. Consequently, the
conformal class of the $L^2$ Hermitian structure on $\mathbb V$ (see
\eqref{kf3}) is preserved by the above involution $\iota''$ of
$\mathbb V$. On the other hand, $\iota(0)\,=\, 0$. Therefore,
$\iota''$ preserves the line ${\mathbb S}\, \subset\, \mathbb V$ in
\eqref{e8}.  Since $\iota ' $ acts as the identity on $\mathbf L_0$,
we deduce that $\iota'' $ is the identity on $\mathbb S$.  In view of
this, and the fact that
\begin{equation}\label{ed3}
\phi\circ f_X\,\,=\,\, \iota\circ\phi\, ,
\end{equation}
where $\phi$ and $f_X$ are the maps in \eqref{e9} and \eqref{e14} respectively, the
lemma follows from Proposition \ref{prop2}.
\end{proof}

Set
\begin{equation*}
\mathbb V_{00} \,: =\, \{ s\,\in\, \mathbb V \,\,\big\vert\,\, m_0 (s) \,\geq\, 4\},
\end{equation*}
where $m_0(s)$ is the multiplicity of $s$ at the origin ${\mathcal O}_X$ of $J^0(X)$. Let
\begin{gather}\label{pht}
\begin{gathered}
\phi^*\,\, :\,\, {\mathbb V}\,\, \longrightarrow\,\, H^0(S,\, \phi^*{\mathbf L})\,\,
\cong\,\, H^0(S,\, K_S(2\Delta))
\end{gathered}
\end{gather}
be the pullback homomorphism for the map $\phi$ in \eqref{e9} (see
Proposition \ref{prop1} for the above isomorphism). Then we have
\begin{gather}\label{eq:2}
\mathbb V_{00}\,\, =\,\, \ker \diffmap^*;
\end{gather}
this is \cite[p.~625, Remark]{vgvg}. Note that we have
\begin{gather}\label{eq:3}
\dim \mathbb V_{00}\,\,=\,\, 2^g - \frac{g(g+1)}{2} -1.
\end{gather}
(see \cite[Proposition 1.1]{vgvg}) and
\begin{equation}\label{ed}
\dim {\rm image}(\diffmap^*)\,\,=\,\, \frac{g(g+1)}{2}+1.
\end{equation}

Set
\begin{gather*}
S\,\,=\,\,X\times X.
\end{gather*}
The two line bundles $K_S$ and $K_S(2\Delta)$ are both equivariant under
$f_X$ in \eqref{e14}. Set
\begin{gather}
H^0(S,\,K_S)^+ \,:= \,\{\alpha \,\in\, H^0(S,\,K_S)\,\big\vert\,\,
f_X^*\alpha \,=\, \alpha\}
\,\cong\, \bigwedge\nolimits^2 H^0(X,\,K_X), \nonumber\\
H^0(S,\,K_S)^-\,:=\, \{\alpha \,\in\, H^0(S,\,K_S)\,\big\vert\,\,
f_X^*\alpha \,=\, -\alpha\} \,\cong\, S^2H^0(X,\,K_X), \label{ka2}\\
H^0(S,\,\lb)^+\,\,:=\, \,\{\alpha \,\in\, H^0(S,\,\lb)\,\,\big\vert\,\, f_X^*\alpha \,=\, \alpha\}, \nonumber\\
H^0(S,\,\lb)^-\,\, :=\,\, \{\alpha \,\in\,
H^0(S,\,\lb)\,\big\vert\,\, f_X^*\alpha \,=\, -\alpha\}.\label{ka4}
\end{gather}
(See \cite[\S~2]{deba} for more details.)

\begin{remark}\label{rem1a}
Proposition \ref{prop1} says that
$\phi^\ast \Thetatwo\,\cong\, p_1^\ast K_X\otimes p_2^\ast
K_X\otimes \mathcal{O}_{S}(2\Delta)$. This line bundle
$p_1^\ast K_X\otimes p_2^\ast K_X\otimes \mathcal{O}_{S}(2\Delta)$
is isomorphic to $K_S(2\Delta)$. Denote by $\lambda$ the natural
isomorphism
\begin{gather*}
\lambda\,:\, H^0(S,\, p_1^\ast K_X\otimes p_2^\ast K_X\otimes
\mathcal{O}_{S}(2\Delta))\,\longrightarrow\, H^0(S,\, K_S(2\Delta)) \\
(p_1^\ast \psi_1)\otimes (p_2^\ast \psi_2)\,\longmapsto\,
(p_1^\ast\psi_1)\wedge (p_2^\ast \psi_2).
\end{gather*}
With a minor abuse of notation, let
\begin{gather*}
f_X^\ast\,:\, H^0(S,\, p_1^\ast K_X\otimes p_2^\ast K_X\otimes
\mathcal{O}_{S}(2\Delta))\,
\longrightarrow\, H^0(S,\, p_1^\ast K_X\otimes p_2^\ast K_X\otimes \mathcal{O}_{S}(2\Delta)),\\
f_X^\ast \,:\,H^0(S,\, \lb)\,\longrightarrow\, H^0(S,\, \lb)
\end{gather*}
denote the involutions induced by the map $f_X$ in \eqref{e14}. Then
we have
\begin{equation}\label{ed2}
f_X^\ast\circ \lambda \,\,=\,\,-\,\lambda \circ f_X^\ast.
\end{equation}
\end{remark}

\begin{lemma}\mbox{}
\begin{enumerate}
\item The image of the map
\begin{gather*}
\diffmap^* \,: \,H^0(J^0(X),\, \Thetatwo) \,\lra\, H^0(S,\, \lb)
\end{gather*}
(see \eqref{pht}) coincides with $H^0(S,\, \lb)^-$ in \eqref{ka4}.

\item The subspace $V_0\,\subset\, H^0(J^0(X),\,\Thetatwo)\, =\, {\mathbb V}$
(defined in \eqref{v0}) maps surjectively to $H^0(S,\,K_S)^-$ in \eqref{ka2}.
\end{enumerate}
\end{lemma}

\begin{proof}
It follows from \eqref{ed3} that for any section $u$ of $\Thetatwo$,
the section $\phi^\ast u$ of $
p_1^\ast K_X\otimes p_2^\ast
K_X\otimes \mathcal{O}_{S}(2\Delta)$ is
$f_X^\ast$-invariant. In view of \eqref{ed2}, this implies that the
section of $\lb$ corresponding to $\phi^\ast u$ is
$f_X^\ast$-anti-invariant. This proves that the image of
$\diffmap^*$ is contained in $H^0(S,\,\lb)^-$.

Next recall that $h^0(S,\,K_S(2\Delta)) \,= \,g^2 + 1 = h^0(S, K_S) + 1$ \cite[Lemma
  2.11]{BR1} and that the section $\omega_\alpha$ of \cite[Corollary
  2.7]{BR1} is anti-invariant with respect to $f_X$.  (The same holds for
  the form $\eta$ that is recalled in Section \ref {section:6}.)
  Hence
  \begin{gather*}
h^0(S,\,K_S(2\Delta))^- \,\,=\,\, \frac{g(g+1)}{2} +1.
  \end{gather*}
Note that this coincides with the dimension of the image of $\phi^*$
(see \eqref{ed}), and hence the first statement of the lemma follows.

{}From the fact that all the sections of $\Thetatwo$ are invariant by
the involution $\iota$ in \eqref{io} it also follows that
\begin{gather*}
V _0 \,= \,\{ s\,\in\, H^0(J^0(X),\,\mathbf{L})\,\, \big\vert\,\, m_0 (s) \,\geq\, 2\},
\end{gather*}
  where $m_0$, as before, is the multiplicity at the origin of
  $J^0(X)$. Note that each nonzero section of $V_0$ vanishes at $0$ to order 2,
and hence the pullback of such a section to $S$ has no pole along $\Delta$. In other words, we
  have
  \begin{equation}\label{eq:4}
\phi^* (V_0) \,\,\subseteq\,\, H^0(S,\, K_S)^-.  
  \end{equation}
  Since the two subspaces $V_0\,\subset\, H^0(J^0(X),\, \Thetatwo)$ and
  $H^0(K_S)^-\,\subset \,H^0(S,\, K_S(2\Delta))^-$ both have codimension
  $1$, we conclude from (1) that in \eqref{eq:4} actually the equality holds.
\end{proof}

\section{Pullback of special sections of $\mathbf L$}\label{sec:pullback:special}

The isomorphism of line bundles in Proposition \ref{prop1} is
determined uniquely up to multiplication by a nonzero constant
scalar. In this Section we introduce a particular class of sections of
$\mathbf L$; see \eqref{def:sw} below. We compute the pullback of
these sections by the map $\phi$ using the formalism of theta
functions.

We start by recalling some notation.  Fix a marking on $X$, i.e., a
symplectic basis $\{a_i,\, b_i\}_{i=1}^g$ of $H_1(X,\,\mathbb Z)$. Let
$\{\omega_i\}_{i=1}^g$ be holomorphic 1-forms on $X$ such that
\begin{equation}\label{nf}
  \int_{a_i}\omega_j \,\,=\,\, \delta_{ij}.
\end{equation}
The period matrix $\tau$ is defined to be
\begin{equation}\label{epm}
  \tau_{ij} \ := \ \int_{b_j}\omega_i.
\end{equation}
We have identifications $H^0(X,\,K_X)^* \,\cong\, \mathbb C^g$ and
$J^0(X) \,\cong\, \mathbb C^g /( \mathbb Z^g + \tau \mathbb Z^g)$. Fix
a universal cover
\begin{gather*}
  \pi\,\,:\,\, \widetilde{X}\,\, \lra\,\, X.
\end{gather*}
Denote by $\tphi$ the lift, to a map between the universal covers, of the difference
map $\phi$ in \eqref{e9}:
\begin{equation}
  \label{phitilde}
  \begin{tikzcd}
\widetilde{X}\times \widetilde{X} \arrow{r}{\widetilde{\phi}} \arrow{d} & \mathbb{C}^g \arrow{d} \\
X\times X \arrow{r}{\phi} & J^0(X)
  \end{tikzcd}
  \qquad
  \widetilde{\phi} (x,\,y) \, =\, \left(\int_{y}^{x}
\omega_1,\,\cdots,\, \int_{y}^{x} \omega_g \right).
\end{equation}
Evidently, we have
\begin{gather}
  \label{tphi:minus}
  \tphi(x,\,x)\,=\,0, \qquad \tphi(x,\,y)\,=\, - \tphi(y,\,x).
\end{gather}
Note that given two points $x,\, y\,\in\, X$, there is always a
contractible open subset $U\,\subset\, X$. By the
uniformization theorem $U$ is biholomorphic to the disk or
$\mathbb C$, so there is a
holomorphic chart $z\,:\, U \,\lra\, \mathbb C$.  Also, there is a
section $s\,:\, U \,\lra\, \widetilde X$ of $\pi
$ which allows us to identify $U$ with the image $s(U)$. Then
\begin{gather*}
  (z_1\,:=\,z\circ p_1,\,\, z_2 \,:=\, z\circ p_2) \ :\ U\times U
\ \longrightarrow\ {\mathbb C}\times{\mathbb C}
\end{gather*}
where $p_i$ is the projection in \eqref{ep}, is a holomorphic coordinate chart on $U\times U
\,= \, s(U)\times s(U)$. Often it is convenient to transfer a computation from $U\times U$
to $s(U)\times s(U)$. By the above observation,
that given two points $x,\, y\,\in\, X$, there is a
contractible open subset $U\,\subset\, X$ containing both of them,
these subsets $U\times U$ cover $X\times X$.

Recall the notation for theta functions:
\begin{gather}\label{tf}
  \langle x,\, y\rangle \,:= \,\sum_{i=1}^n x_iy_i, \qquad \notag
  \theta(z)\,\,:=\,\, \sum_{m\in \mathbb{Z}^g} e^{\pi\sqrt{-1}\langle m, \tau
m + 2 z\rangle} .
\end{gather}
Let $\Theta$ denote the divisor of $\theta$. Any $\zeta\,\in\, \mathbb{C}^g$ can be expressed
uniquely as $\zeta\,=\tau a+\, b$ for $a,\,b\,\in\,
\mathbb{R}^g$. Then one defines
\begin{gather}\label{tzeta}
  \theta[\zeta](z) \,:=\, \sum_{m\in \mathbb{Z}^g}
  e^{\pi\sqrt{-1}\langle m+a, \tau(m+a) + 2 (z+b)\rangle}\,=\, e^{\pi
\sqrt{-1} \langle a, \tau a + 2 (z+b) \rangle } \cdot
  \theta(z+\zeta).
\end{gather}
So $\theta[\zeta] (z ) \,=\,0$ if and only if we have
$\theta(z+\zeta) \,=\, 0$. Also, we have $\theta(z+\zeta)\, =\, 0$ if
and only if $\zeta+z\,\in\, \Theta$. As usual we identify theta
functions on $\mathbb C^g$ with sections of line bundles over
$J^0(X)$. So $\theta$ in \eqref{tf} is identified with a section of a symmetric
line bundle, whose square is $\mathbf L$, while $\theta[\zeta]$ in \eqref{tzeta} is
identified with a section of a translation of this line bundle.

For $\zeta \,\in\, J^0(X)$, define
\begin{equation}\label{def:sw}
  s_\zeta(z) \,\,:=\,\, \theta(z-\zeta) \theta(z+\zeta).
\end{equation}
By the Theorem of the Square we have $s \,\in\, H^0(J^0(X),\,\mathbf L)$. It
is known that $\{s_\zeta\}_{\zeta\in\mathbb{C}^g}$ is a set of
generators of $H^0(J^0(X),\,\Thetatwo)$ \cite{Fay,Gu3}. These sections
$s_\zeta$ are the special sections in the title of this Section \ref{sec:pullback:special}.

The following property of $\theta$ will be very useful here.

\begin{lemma}\label{allthetalemma}
Take any $\zeta\,\in\,\Theta$. Write as usual $\zeta\,=\,\tau a+b$, and
  define
  $$c(\zeta)\,:=-\,e^{2\pi \sqrt{-1}\langle a, \tau a+2b \rangle}.$$
  Then the following holds:
  \begin{gather*}
 2\,\frac{\partial \theta[\zeta]}{\partial z_i}(0) \,
 \frac{\partial \theta[\zeta]}{\partial z_j}(0)
 \,\,=\,\,c(\zeta)\frac{\partial^2 s_\zeta}{\partial z_i\partial
 z_j}(0).
  \end{gather*}
\end{lemma}

\begin{proof}
  Denote by $u_i$ the $i$-th partial derivative of a function $u$.
  Plugging the value $-z$ in \eqref{tzeta}, and using that $\theta$ is
  even, we conclude that
  \begin{gather*}
 \theta[\zeta](-z)\,\,=\,\,e^{\pi \sqrt{-1}\langle a,\tau
 a+2(b-z)\rangle}\theta(z-\zeta).
  \end{gather*}
  Multiplying this equation by \eqref{tzeta} yields
  \begin{equation}\label{eq:2sectionzeta}
\theta[\zeta](z)\theta[\zeta](-z) \,\,=\,\,-\,c(\zeta)s_\zeta(z).
  \end{equation}
  The lemma follows by evaluating  at $z=0$ the double derivatives of both
  sides.  Indeed, taking double derivatives,
  $$
  \frac{\partial^2 \theta[\zeta](z)\theta[\zeta](-z) }{\partial
z_i\partial z_j}\Big\vert_{z=0}\,\,=\,\,
  2\theta[\zeta]_{ij}(0)\theta[\zeta](0)-2\theta[\zeta]_i(0)\theta[\zeta]_j(0)
  \,\,=\,\,-2\theta[\zeta]_i(0)\theta[\zeta]_j(0)$$ because
  $\zeta\,\in\,\Theta$, on the other hand, the second derivative of the right hand
  side of \eqref{eq:2sectionzeta} evaluated at $0$ is clearly
  $-c(\zeta)(s_\zeta)_{ij}(0)$.
 \end{proof}

 Our goal now is to compute the pullback of the special sections
$s_\zeta$, for $\zeta \in \Theta_{\rm reg}$. Fix a point $x_0 \,\in\, X$, and let
\begin{gather*}
A\,\,:\,\, X \,\,\lra\,\, J^0(X)
\end{gather*}
be the Abel-Jacobi map with $x_0$ as the base point. Hence we have
$\phi(x,\,y) \,=\, A(x) - A(y)$, where $\phi$ is the map in \eqref{e9}.

\begin{definition}\label{def:Fy}
For any $\zeta \,\in\, \Theta$ and $y\,\in\, X$, define
        \begin{gather*}
          F_{y,\zeta}( x )\,=\, \theta (A(x) -A(y) - \zeta) \,=\,
          \theta (\phi(x,\,y)- \zeta),
        \end{gather*}
        which is a section of a line bundle on $X$; it is also a holomorphic
        function on the universal covering of $X$.
\end{definition}

The following lemma is well-known.

\begin{lemma}\label{narasimhan111}\mbox{}
\begin{enumerate}
\item [(a)] There is some point $y_0\,\in\, X$
such that $F_{y_0,\zeta}\,\not\equiv\, 0$ holds if and only if
$ \zeta \,\in\, \Theta_{\rm reg}$ (the regular locus of the divisor).

\item [(b)] If $\zeta \,\in\, \Theta_{\rm reg}$, then there is an
effective divisor $D_\zeta$ of degree $g-1$, with
$\dim |D_\zeta|\, =\,0$, such that the following holds:
\begin{gather*}
\operatorname{div} F_{y,\zeta} \,=\, y + D_\zeta
\end{gather*}
for all $y\,\in\, X$ for which $F_{y,\zeta}\,\not\equiv\, 0$.

\item [(c)] There is $\kappa \,\in \,J^0(X)$ satisfying the
condition $2\kappa\,=\,-A(K_X)$, such that the equality $A(D_\zeta)\,=\,\zeta-\kappa$ holds.
\end{enumerate}
\end{lemma}

\begin{proof}
See \cite[p.~112, Theorem 1]{narasimhan-Riemann} for (a) and
\cite[p.~111, Lemma 2]{narasimhan-Riemann} for (b). Statement
(c) follows immediately from \cite[p.~97, Theorem 1]{narasimhan-Riemann}
applied to $A(y) + \zeta$.
\end{proof}

\begin{lemma}\label{lemc}

For any $\zeta \,\in\, \Theta_{\rm reg}$, the divisor $D_\zeta + D_{-\zeta}$ (see
Lemma \ref{narasimhan111}(b)) is a canonical divisor of $X$.
\end{lemma}

\begin{proof}
We have
\begin{gather*}
A(D_\zeta ) \ =\ \zeta - \kappa, \\
A(D_{-\zeta} ) \ =\ -\zeta - \kappa ,\\
A (D_\zeta + D_{-\zeta} ) \ = \ -2 \kappa = A(K_X) .
\end{gather*}
Therefore, $D_\zeta + D_{-\zeta}$ is linearly equivalent to a canonical divisor, and hence
it is a canonical divisor.
\end{proof}

\begin{definition}\label{def:omega:zeta}
  For $\zeta \,\in\, \Theta$, set
  \begin{equation}
    \label{eq:omega:zeta}
    \omega_\zeta\,\,:=\,\, \sum_{i=1}^g \desu[\theta{[\zeta]}]{z_i} (0)\cdot
    \omega_i \,\,\in\,\, H^0(X,\,K_X).
  \end{equation}
\end{definition}

\begin{lemma}\label{lemma:omega:zeta}
If $\zeta \,\in\, \Theta_{\rm reg}$, then
$\operatorname{div} \omega_\zeta \,=\, D_\zeta + D_{-\zeta}$.
\end{lemma}

\begin{proof}
First it will be shown that
  $D_\zeta \,\leq\, \operatorname{div} \omega_\zeta$. Assume that $k\cdot x_0 \,\leq  \, D_\zeta$.
This means that $F_{y,\zeta}\,=\, \theta (A(x) -A(y) - \zeta) $ has a zero of order at least $k$ at $x\,=\,x_0$.
As $\theta$ is an even function, it follows that
$ \theta(A(y) -A(x) +\zeta) $, and hence $ \theta[\zeta](A(y) -A(x)) $,
also has a zero of order at least $k$ at $x\,=\,x_0$.

Set, for simplicity,
\begin{gather*}
f(x,y)\ :=\ \theta[\zeta](A(y) -A(x))  .
\end{gather*}
For a generic $y \,\in\, X$, we have
$f(\cdot, \,y) \,\not\equiv\, 0$ and $f(\cdot,\, y)$ has a
zero of order at least $k$ at $x\,=\,x_0$. Therefore, locally we have
\begin{gather*}
  f(x,\,y)\ =\ (x-x_0)^k h(x,y)
\end{gather*}
for some holomorphic function $h$. Differentiating with respect to $y$
we find that $\partial f /\partial y$ also has a zero of order at least $k$ at $x\,=\,x_0$:
\begin{gather*}
  \desu[f]{y} (x,\,y)\ =\ (x-x_0)^k \desu[h]{y}(x,\,y).
\end{gather*}
In particular, restricting to the diagonal, we conclude that the function
    \begin{gather*}
      t\ \longmapsto\ \desu[f]{y} (t,\,t)
  \end{gather*}
has a zero of order at least $k$ at $t\,=\,x_0$.
Now differentiating with respect to $y$ we get the following:
\begin{gather*}
\desu[f]{y} (x,\,y)\ =\ \desu{y}
\ \theta[\zeta](A(y)-A(x)) \ =\ \sum_i \desu[\theta{[\zeta]}]{z_i} (A(y)-A(x)
)\cdot\omega_i (y) .
  \end{gather*}
  Hence
  \begin{gather*}
    \desu[f]{y}(t,t) = \,\, \sum_i \desu[\theta{[\zeta]}]{z_i} (0
    )\cdot \omega_i (t) = \omega_\zeta(t),
  \end{gather*}
in particular, $\omega_\zeta$ has a zero of order at least $k$ for $t\,=\,x_0$.
This implies that $k\cdot x_0 \,\leq\, \operatorname{div} \omega _\zeta$,
and thus we have $D_\zeta \,\leq\, \operatorname{div} \omega _\zeta$.  So
\begin{equation*}
D_\zeta + E \,\,=\,\, \operatorname{div} \omega _\zeta \,\,=:\,\,K',
\end{equation*}
where $E$ is effective and $K'$ is evidently a canonical
divisor. But $K''\,:=\,D_\zeta + D_{-\zeta}$ is also a canonical
divisor by Lemma \ref{lemc}. Consequently, $K' \,\sim\, K''$, and therefore
  \begin{equation*}
    E \ \sim\ D_{-\zeta}.
  \end{equation*}
  But by Lemma  \ref {narasimhan111}(b) we have $\dim|D_{-\zeta}|\,=\,0$. Hence 
  $E \,=\, D_{-\zeta}$ and 
  \begin{gather*}
    \operatorname{div} \omega _\zeta\ =\ D_\zeta + D_{-\zeta}.
  \end{gather*}
  This completes the proof.
\end{proof}

\begin{proposition}\label{prop:omega:zeta}\mbox{}
\begin{enumerate}
\item [(a)] For any $\zeta \,\in\, \Theta$,
$$s_\zeta(0) \,=\, 0.$$

\item [(b)] For any $\zeta \,\in\, \Theta_{\rm reg}$, the sections
$\phi^* (s_\zeta )$ and $p_1^*\omega_\zeta \wedge p_2^* \omega_\zeta$
have the same divisor ($s_\zeta$ and $\omega_\zeta$ are defined in
\eqref{def:sw} and \eqref{eq:omega:zeta} respectively, and $p_i$
is the projection in \eqref{ep}).
\end{enumerate}
\end{proposition}

\begin{proof}
  Fix $\zeta \,\in\, \Theta$.  It follows from
  \eqref{def:sw} that
  $s_\zeta (0) \,=\, \theta(\zeta) \theta(-\zeta) \,=\,
  \theta(\zeta)^2$.  This is zero because $\zeta \,\in\, \Theta$.  So
  (a) is proved.

Assume now that $\zeta \,\in\, \Theta_{\rm reg}$.
Set
\begin{gather*}
B_+\,\,:=\,\,\{y\,\in\, X\,\,\big\vert\,\, F_{y,\zeta}
\,\not\equiv\, 0\}, \ \, B_{-}\,\,:=\,\,\{y\,\in\, X\,\,\big\vert\,\, F_{y,-\zeta}
\,\not\equiv\, 0\}.
\end{gather*}
  $B_+$ and $B_-$ are Zariski open subsets of $X$. Set
  \begin{gather*}
    B\,\,:=\,\,B_-\cap B_+.
  \end{gather*}

  Since $\zeta \,\in\, \Theta_{\rm reg}$, and $\Theta$ is symmetric,
  we have $-\zeta \,\in\, \Theta_{\rm reg}$. It follows from Lemma
  \ref{narasimhan111} that $B_-$, $B_+$ and $B$ are nonempty. Note that
  $\phi(x,\,y) \,= \,A(x) - A(y) $, and hence using \eqref {def:sw} we get that
  \begin{gather}
    \label{eq:6}
    s_\zeta (\phi (x,\,y)) \,=\, \theta (A(x ) - A(y) - \zeta) \cdot
    \theta (A(x ) - A(y) + \zeta) \,=\, F_{y,\zeta}(x) \cdot
    F_{y,-\zeta}(x)
  \end{gather}
  By Lemma \ref{narasimhan111}, for any $y \,\in\, B$, we have
  \begin{gather*}
    \operatorname{div} F_{y,\zeta} \,=\, y + D_\zeta, \qquad
    \operatorname{div} F_{y,-\zeta} \,=\, y+ D_{-\zeta}.
  \end{gather*}
  This and \eqref{eq:6} together yield
  \begin{gather*}
    \operatorname{div} (\phi^*s_\zeta) \vert_{X\times \{y\}}\,\,=\,\,
    2 y + D_\zeta + D_{-\zeta} .
  \end{gather*}
  Using Lemma \ref{lemma:omega:zeta} we conclude that for any
  $y\,\in\, B$,
  \begin{gather*}
    \operatorname{div} (\phi^*s_\zeta) \vert_{X\times \{y\}} \,\,=\,\,
    2y + \operatorname {div} \omega_\zeta .
  \end{gather*}
  Hence
  \begin{gather*}
    \operatorname{div} (\phi^*s_\zeta) \vert_{X\times B} \,\,=\,\,
    2\Delta + p_1^* \operatorname{div} \omega_\zeta.
  \end{gather*}
  Since $\theta$ is even,
  \begin{gather*}
    s_\zeta (\phi (x,\,y)) \,\,= \,\, s_\zeta (\phi (y,\,x)).
  \end{gather*}
So, for $x \,\in\, B$,
\begin{gather*}
\operatorname{div} (\phi^*s_\zeta) \vert_{\{x\}\times X}\,=\,
\operatorname{div} (\phi^*s_\zeta) \vert_{X \times \{x\}}\,=\, 2 x
+ D_\zeta + D_{-\zeta} \,=\, 2x + \operatorname {div} \omega_\zeta .
\end{gather*}
Thus
\begin{gather*}
\operatorname{div} (\phi^*s_\zeta) \vert_{B\times X} \,\,=\,\,
2\Delta + p_2^* \operatorname{div} \omega_\zeta.
\end{gather*}
Hence on
$U\,:=\,(X\times B )\, \cup \, (B\times X) \,=\, S - (X-B) \times
(X-B)$, we have
\begin{gather*}
\operatorname{div} (\phi^*s_\zeta) \vert_{U} \,=\, 2\Delta + p_1^*
\operatorname{div} \omega_\zeta+ p_2^* \operatorname{div}
\omega_\zeta \,=\, 2\Delta + \operatorname{div} \left( p_1^*
\omega_\zeta\wedge p_2^* \omega_\zeta \right ) .
\end{gather*}
  (Note that this shows that the support of
  $p_1^*\operatorname{div} \omega_\zeta$ is contained in $X -B$. This
  can be checked explicitly; see Remark \ref{Bpm} below.)  As $X-B$ is
  a finite set, the complement of $U$ in $S$ is finite, and hence
\begin{gather}\label{divisor:divisor}
    \operatorname{div} (\phi^*s_\zeta) \,=\, 2\Delta +
    \operatorname{div} \left( p_1^* \omega_\zeta\wedge p_2^*
      \omega_\zeta \right ).
  \end{gather}
This completes the proof.
\end{proof}

The right-hand side in \eqref{divisor:divisor} is the divisor of
$ p_1^* \omega_\zeta\wedge p_2^* \omega_\zeta $ viewed as a section of
$K_S(2\Delta)$. Hence Proposition \ref{prop:omega:zeta}(b) gives an
independent proof of Proposition \ref{prop1}.

\begin{remark}
  \label{Bpm} From the proof of Proposition \ref{prop:omega:zeta} one
  might be inclined to guess that $B\,=\,X - \operatorname{supp} (D_\zeta +
  D_{-\zeta})$. This is indeed true. In fact we claim that
  $B_+ \,=\, X -  \operatorname{supp} D_{-\zeta}$ and $B_-\,=\,X -  \operatorname{supp} D_\zeta$. To prove this,
  observe that
  \begin{gather*}
    F_{y,\zeta}(x) \,=\, \theta(A(x) -A(y) - \zeta) \,=\, \theta(A(y)
    -A(x) + \zeta)\, =\, F_{x, -\zeta}(y).
  \end{gather*}
  Now, if $y \,\not\in\, B_+$, then $F_{y,\zeta}(x) \,=\,0$ for all
  $x$, so $F_{x,-\zeta}(y) \,=\,0$ for all $x$. So $y$ belongs to the
  support of the divisor $x + D_{-\zeta}$ for all $x $. Hence we have
  $y\,\in\, \operatorname{supp} D_{-\zeta}$. Conversely, if
  $y\,\in\, \operatorname{supp} D_{-\zeta}$, then $y$ belongs to the
  support of $x + D_{-\zeta}$ for all $x$, and hence
  $F_{x,-\zeta}(y) \,=\,0$ meaning $F_{y,\zeta } \,\equiv\, 0$.  This
  proves the claim.  It follows immediately that
  $B\,=\,X - \operatorname{supp} (D_\zeta + D_{-\zeta})$.
\end{remark}

\section{Pullback of arbitrary sections of $\mathbf L$}
\label{section:5}

Out next step is to compute the pullback of an arbitrary section of
$\mathbf L$.  This is achieved using second order theta functions as
developed in the unpublished book \cite{Gu3}. In particular we use
Fay's Trisecant Formula; see Theorem \ref {thm:trisecant} below.

Recall that a characteristic $\zeta \,\in\, \mathbb C^g$ is called
\emph{half-integer} if
$ 2\zeta\,\in\, \mathbb{Z}^g+\tau\mathbb{Z}^g $, while it is called
\emph{odd} (respectively, \emph{even}) if $\theta[\zeta]$ is an odd
(respectively, even) function. It is straightforward to see that any
even or odd characteristic is half-integer \cite[p.~20, A]{Gu3}.

An odd characteristic $\zeta$ is \emph{regular} if not all partial
derivatives of $\theta[\zeta]$ vanish at the origin. Such
characteristics exist on any smooth projective curve (see \cite[p.~16,
Remark]{Fay}).

A direct computation using \eqref{tzeta} shows that for any odd
characteristic $\zeta$ we have
\begin{equation}
  \label{szeta}
  \theta[\zeta]^2\,\,=\,\, c(\zeta) \cdot s_\zeta, \qquad  c(\zeta) \,\,=\,\, e^{2\pi \sqrt{-1}\langle
    a, \tau a \rangle},
\end{equation}
where $\zeta \,=\, \tau a + b$. Notice that from the fact that $\theta[\zeta]$ is odd
it follows immediately that we have $\zeta \,\in\, \Theta$.

Several classical constructions and results are concerned with odd
regular characteristics. We will now recall two of them that are
important for our purpose. If $\alpha \,\in\, H^0(S,\,K_S(2\Delta))$
and $c \,\in\, H_1(X,\,\mathbb Z)$, then one can define the integrals
\begin{gather*}
  \int_{p\in c} \alpha (p,\,q) \qquad \text{and} \qquad \int_{q\in c}
  \alpha (p,\,q).
\end{gather*}
These are holomorphic 1-forms on $X$.  See \cite[p.~115]{Gu2} for
details.

\begin{proposition}\label{canonical:bidifferential}
Let $\{a_i,\, b_i\}$ be a symplectic basis of $H_1(X,\,\mathbb Z)$,
and let $\omega_i$ be the normalized differentials as in \eqref{nf}. Then there is a unique
$\Omega \,\,\in\,\, H^0(S,\, K_S(2\Delta))^-$ (see \eqref{ka4}) that
satisfies the period conditions
\begin{gather*}
\int_{q\in a_k} \Omega(p,\,q)\,=\,0,\qquad \int_{q\in b_k}
\Omega(p,\,q)\,=\,2\pi \sqrt{-1}\omega_k(p).
\end{gather*}
This $\Omega$ is called the \emph{canonical bidifferential} or
\emph{double differential} of $X$ with respect to the symplectic
basis $\{a_i,\,b_i\}$. It has a double pole on the diagonal with
biresidue 1. Let $\zeta$ be an odd regular characteristic, and let
$z\,:\,U\,\longrightarrow\, \mathbb{C}$ be a holomorphic coordinate
function on a contractible subset $U\,\subset\, X$. Set
$z_i\,=\,p_i^\ast z$. Identify $U$ with one of its liftings to
the universal cover $\widetilde X$ of $X$, and set
  \begin{gather*}
    \due\,:\, U\times U \,\longrightarrow\, \mathbb C, \qquad \due
    (z_1,\, z_2)\,:=\, \theta[\zeta](\tphi( s(z_1),\,s( z_2) ).
  \end{gather*}
  Then
  \begin{gather}\label{canonical:log}
    \Omega\,\ =\,\ \frac{\partial^2 \log \due}{\partial z_1\partial
      z_2} \, dz_1\wedge dz_2.
  \end{gather}
  In particular, this derivative is independent of the section $s$.
\end{proposition}

\begin{proof}
See \cite[p.~20, Corollary 2.6]{Fay} and \cite[p.~125, Theorem
4.23]{Gu2}. In the former reference $\Omega$ is denoted by $\omega$
and in the latter by $\widehat{\mu}_M$. To prove
\eqref{canonical:log} one needs a different definition for $\Omega$.
Let $E$ denote the prime form of $X$. This is a section of
$\mathcal O_S (\Delta)$. For each regular odd half-period $\zeta$, the corresponding
theta characteristic $L_\zeta$ has a section $h_\zeta$ such that
$h_\zeta ^2 \,=\,  \omega_\zeta$ (in \cite{Fay}, $\omega_\zeta$ is called
$H_\zeta$). Then the following holds:
\begin{gather*}
E (z_1, z_2) \,\,=\,\, \frac{\beta(z_1, z_2) } {
h_\zeta (z_1) \cdot h_\zeta(z_2)};
  \end{gather*}
  see \cite [p.~16, Definition 2.1]{Fay}. On the other hand, by formula (28) in
  \cite[p. 20]{Fay},
  \begin{gather*}
    \Omega \,\,=\,\,\frac{\partial^2 \log E}{\partial
      z_1\partial z_2} \, dz_1\wedge dz_2.
  \end{gather*}
  From this \eqref{canonical:log} follows immediately.
\end{proof}

In the following we use the map  $\widetilde\phi$ constructed in \eqref{phitilde}.

\begin{theorem}[{Fay's Trisecant Identity}]
  \label{thm:trisecant}
  Let $\zeta$ be a regular odd characteristic. Let
  $w\,\in\, \mathbb{C}^g$, and let $z_1,\,z_2,\,a_1,\,a_2$ be points of
  the universal cover $\widetilde{X}$ of $X$. Then the following
  holds:
  \begin{equation}\label{trisecant}
    \begin{gathered}
      \theta (w + \tphi (z_1, \,z_2) ) \cdot \theta (w + \tphi (a_2,\,
      a_1) ) \cdot \theta[\zeta] ( \tphi(z_1,\, a_1 )) \cdot
      \theta[\zeta]( \tphi(z_2, \,a_2 ))
      \\
      =\, \theta (w ) \cdot \theta (w + \tphi (z_1, \,z_2) +\tphi
      (a_2, \,a_1) ) \cdot \theta[\zeta]
      ( \tphi(z_1, \,a_2 )) \cdot  \theta[\zeta] (  \tphi(z_2,\, a_1 )) +\\
      + \theta (w + \tphi (z_1, \,a_1) ) \cdot \theta (w + \tphi
      (a_2,\, z_2) ) \cdot \theta[\zeta] ( \tphi(z_1,\, z_2 )) \cdot
      \theta[\zeta] ( \tphi(a_1, \,a_2 )).
    \end{gathered}
  \end{equation}
\end{theorem}

\begin{proof}
  Fay's Trisecant Identity takes many forms. To get the equality in \eqref{trisecant}, start
  from equation (1-1) in \cite{poor}. Note that $\theta[\zeta]$ is odd,
and hence \eqref {tphi:minus} implies that
  $\theta[\zeta ] (\tphi(a_1,\,z_2)) \,= \,- \theta[\zeta ]
  (\tphi(z_2,\,a_1))$. Since $\zeta$ is a regular point of $\Theta$,
  equation (1-2) from the same paper holds true for
  $\alpha \,=\, \zeta$. Substituting this, we get rid of the
  cross-ratio function.  Using \eqref{tzeta} one replaces
  $\theta (\zeta + \cdot)$ with $\theta[\zeta]$ and the exponentials
  cancel out. Finally clearing denominators we arrive at
  \eqref{trisecant}.
\end{proof}

Let $u_i$ and $u_{ij}$ denote the partial derivatives of a function
$u$ on $\mathbb C^g$.

\begin{proposition}
  \label{prop:pullback}
  If $f\,\in\, H^0(J^0(X),\,\Thetatwo)$, and $\zeta$ is a regular odd
  characteristic, then
  \begin{equation}
    \label{eq:pullback}
    f\circ \tphi\,\cdot \, p_1^\ast \omega_\zeta\wedge p_2^\ast \omega_\zeta \,\,= \,\,\theta[\zeta]^2\circ \tphi\, \cdot \,\left(
      f(0)\, \Omega+\frac{1}{2} \sum_{i,j=1}^g  f_{ij}(0)\, p_1^\ast \omega_i\wedge p_2^\ast \omega_j\right),
  \end{equation}
  where $\Omega$ is the canonical bidifferential as in Proposition
  \ref{canonical:bidifferential}.
\end{proposition}

\begin{proof}
  Choose $w\,\in\, \mathbb C^g$, and set $f\,:=\, s_w$. Since the collection of
  sections $s_w$, with $w$ varying over $\mathbb C^g$, generate
  $H^0(J^0(X),\,\mathbf L)$, it is enough to prove \eqref{eq:pullback}
  for these sections.

  To do this, we start from Fay's Trisecant Identity.  Set
  \begin{gather*}
    \begin{gathered}
      z\,:=\,(z_1,\, z_2), \\ a\,:=\,(a_1,\,a_2),
    \end{gathered}
    \qquad
    \begin{aligned}
      \one (z) &\,:=\, \theta ( w + \tphi (z) ),\\
      \due (z) &\,:=\, \theta [\zeta] (  \tphi (z) ),\\
      \three (z) &\,:=\, \theta ( w + \tphi (z) -\tphi (a) ).
    \end{aligned}
  \end{gather*}
  Note that since $\theta[\zeta]$ is odd, we have
  \begin{gather}
    \due(z_2,\,z_1) \,= \,- \due(z).  \label{poz:odd}
  \end{gather}
  Set
  \begin{equation}
    \label{eq:againtrisecant}
    \begin{split}
      A(z)&\,:=\, \one (z) \cdot \one (a_2, \,a_1) \cdot \due (z_1,\,
      a_1 )
      \cdot \due (z_2, \,a_2 ),\\
      B(z) &\,:=\, \theta(w) \cdot\three (z) \cdot \due(z_1, \,a_2 )
      \cdot \due(z_2, \,a_1) + \one (z_1,\, a_1) \cdot \one (a_2,\,
      z_2) \cdot \due(z ) \cdot \due(a ).
    \end{split}
  \end{equation}
Then \eqref{trisecant} becomes $A\,=\,B$.  The crucial step in the
proof is to compute the second derivatives $B_{12}$ and $A_{12}$
with respect to $z_1$ and $z_2$ at the point $(a_1,\, a_2)$.  This
computation is elementary but long and it is adapted from the proof
of Theorem 6 in \cite[p.~24, part D]{Gu3}.  After some work we
will recognize that the equation
\begin{gather}\label{eq:10}
    A_{12}dz_1 \wedge dz_2 \,=\, B_{12}dz_1 \wedge dz_2
  \end{gather}
  is exactly the same as \eqref{eq:pullback} for $f\,=\,s_w$.

  We start by computing $B_{12}(a)$.  Using \eqref{poz:odd} we get that
  \begin{equation}
    \label{eq:againtrisecant:2}
    \begin{split}
      A(z)&\,:= \,- \one (z) \cdot \one (a_2,\, a_1) \cdot \due
      (z_1,\, a_1 )
      \cdot \due (a_2,\, z_2 ),\\
      B(z)&\,:= \,- \theta(w) \cdot\three (z) \cdot \due(z_1,\, a_2 )
      \cdot \due( a_1,\,z_2) + \one (z_1,\, a_1) \cdot \one (a_2,\,
      z_2) \cdot \due(z ) \cdot \due(a ).
    \end{split}
  \end{equation}
  We first compute the derivative $B_{12}(a)$ in terms of the
  functions $\one ,\, \due$ and $\three$:
  \begin{gather*}
    \begin{aligned}
      B_1 (z) \,=\,&- \theta(w) \cdot\three_1 (z) \cdot \due(z_1,\,
      a_2 ) \cdot \due( a_1,\,z_2) +
      \\
      & - \theta(w) \cdot\three (z) \cdot \due_1(z_1,\, a_2 ) \cdot
      \due(a_1,\,z_2)
      +    \\
      & + \one_1 (z_1,\, a_1) \cdot \one (a_2,\, z_2) \cdot
      \due(z ) \cdot \due(a ) + \\
      &+ \one (z_1, \,a_1) \cdot \one (a_2,\, z_2) \cdot \due_1(z )
      \cdot \due(a ) ,
      \\
      B_{12} (a) \,=\, & - \theta(w) \cdot\three_{12} (a) \cdot \due(
      a ) \cdot \due( a) - \theta(w) \cdot\three_1 (a) \cdot \due( a )
      \cdot \due_2( a) +
      \\
      & - \theta(w) \cdot\three_2 (a) \cdot \due_1( a ) \cdot \due(a)
      - \theta(w) \cdot\three (a) \cdot \due_1( a ) \cdot \due_2( a)
      +    \\
      & + \one_1 (a_1,\, a_1) \cdot \one_2 (a_2,\, a_2) \cdot \due(a
      )^2 + \one_1 (a_1,\, a_1) \cdot \one (a_2, \,a_2) \cdot \due_2(a
      ) \cdot \due(a )
      + \\
      & + \one (a_1,\, a_1) \cdot \one_2 (a_2,\, a_2) \cdot \due_1(a )
      \cdot \due(a ) + \one (a_1,\, a_1) \cdot \one (a_2, \,a_2) \cdot
      \due_{12}(a )
      \cdot \due(a )\\
      =\,&\,\, \theta(w)^2 \bigg ( \due_{12}(a ) \cdot \due(a ) -
      \due_1( a )
      \cdot \due_2( a) \bigg ) +\\
      &+ C \cdot \theta(w) \cdot \due(a) + + \one_1 (a_1,\, a_1) \cdot
      \one_2 (a_2,\, a_2) \cdot \due(a )^2,
    \end{aligned}
  \end{gather*}
  where
$$ C \,=\, - \three_{12} (a) \cdot \due(a ) - \three_1
(a) \cdot \due_2( a) -\three_2 (a) \cdot \due_1(a ) + \one_1 (a_1,\,
a_1) \cdot \due_2(a ) + \one_2 (a_2, \,a_2) \cdot \due_1(a ).$$ Now
assume that on $U$,
\begin{gather*}
  \omega_i \,\,=\,\, h_i(z) dz.
\end{gather*}
Recall from \eqref{phitilde} the definition of $\widetilde \phi$.
We compute some values and derivatives:
\begin{gather*}
  \begin{aligned}
    &      \one_1(a)  \,= \,\theta_k (w + \tphi(a)) h_k(a_1) , \\
    &      \one_2(a)  \,=\,  - \theta_k (w + \tphi(a)) h_k(a_2) , \\
  \end{aligned}
  \qquad
  \begin{aligned}
    &      \one(a_1,\,a_1)  \,=\, \one(a_2,\,a_2) \,=\, \theta(w),\\
    & \three(a) \,=\, \theta(w).
  \end{aligned}
\end{gather*}
Here repeated indices are summed. We compute some more derivatives:
\begin{gather}
  \begin{aligned}
    &      \one_1 (x,\,x ) \,=\, \theta_i(w) h_i(x) , \\
    &      \one_2 (x,\,x ) \,=\, -\theta_i(w) h_i(x),       \\
    &      \one_1 (a_1,\,a_1 ) \,=\, \theta_i(w) h_i(a_1) , \\
    &      \one_2       (a_2,\,a_2 ) \,=\, -\theta_i(w) h_i(a_2),\\
  \end{aligned}
  \qquad
  \begin{aligned}
    &      \one_1 (a_1,\,a_1 ) \,=\, \three_1(a), \\
    &      \one_2 (a_2,\,a_2 ) \,=\, \three_2(a),\\
    &    \three_1(a) \,=\, \theta_i (w) h_i(a_1), \\
    &    \three_2(a) \,=\, -\theta_j (w) h_j(a_2).\\
  \end{aligned}
  \label{derivatejeff}
\end{gather}
Finally some terms cancel out:
\begin{gather*}
  C\,=\, - \three_{12} (a) \cdot \due(a ) - \three_1 (a) \cdot \due_2(
  a) -\three_2 (a) \cdot \due_1(a ) + \three_1(a) \cdot \due_2(a ) +
  \three_2 (a)
  \cdot \due_1(a )\\
  =\,- \three_{12} (a) \cdot \due(a ).
\end{gather*}
So again using \eqref{derivatejeff} we obtain that
\begin{gather*} B_{12}(a) \,=\, \theta(w)^2 \bigg ( \due_{12}(a )
  \cdot \due(a ) - \due_1( a
  ) \cdot \due_2( a) \bigg ) \\
  - \theta(w) \cdot \due(a) \cdot \three_{12} (a) \cdot \due(a )
  + \one_1 (a_1,\, a_1) \cdot \one_2 (a_2, \,a_2) \cdot \due(a )^2 \\
  =\, \theta(w)^2 \bigg ( \due_{12}(a ) \cdot \due(a ) - \due_1( a )
  \cdot \due_2( a) \bigg ) +\due(a)^2 \cdot \bigg( \three_1 ( a)
  \three_2(a) -\theta(w) \cdot \three_{12} (a) \bigg ).
\end{gather*}
Recall that
\begin{gather*}
  (\log \due) _{12}(a) \,\,=\,\, \frac{ \due_{12}(a ) \cdot \due(a) -
    \due_1( a ) \cdot \due_2( a) }{\due(a)^2},\\
  \due(a) \,=\, \theta[\zeta] (\tphi(a)), \qquad \due(a)^2 \,=\,
  \theta[\zeta]^2 (\tphi(a)).
\end{gather*}
Set
\begin{gather*}
  D\,\,:=\,\, \three_1(a) \three_2(a) - \theta(w) \cdot \three_{12}
  (a).
\end{gather*}
Then using \eqref{canonical:log} we get that
\begin{gather*}
  B_{12}(a) \,=\, \due(a)^2\cdot \Big [ \theta (w)^2 \cdot (\log \due)
  _{12}(a) + D
  \Big ]\\
  =\,\, \theta[\zeta]^2 (\tphi(a)) \cdot \Big [ s_w(0) \cdot (\log
  \due) _{12}(a) + D \Big ]
  ,\\
  B_{12}(a) dz_1\wedge dz_2 \,\,=\,\, \theta[\zeta]^2 (\tphi(a)) \cdot
  \Big [ s_w(0) \cdot \Omega + D \, dz_1\wedge dz_2 \Big ].
\end{gather*}
We compute some more derivatives:
\begin{align*}
  &  \notag  \three_1 (z) \,=\, \theta_i (w +\tphi(z) - \tphi(a)) h_i(z_1), \\
  &   
    \three_{12} (a) \,=\, -\theta_{ij} (w ) h_j(a_2)h_i(a_1) , \\
  &    \three_1 (a) \,=\, \theta_i (w ) h_i(a_1), \\
  &  \three_2 (a) \,=\,-  \theta_i (w ) h_i(a_2).
\end{align*}
Recalling that $ f\,=\,s_w$,
\begin{align*}
  &    f(z) \,=\, \theta(z+w) \theta (z-w)\\
  &    f_{ij}(0) \,=\, 2 \theta (w) \theta_{ij}(w) - 2 \theta_i(w) \theta_j(w),\\
  &  D \,=\,
    -\theta_i(w) h_i(a_1) \theta(w) h_j(a_2)
    + \theta(w)\theta_{ij} (w ) h_j(a_2)h_i(a_1)= \\
  &    =\, \left ( \theta(w)\theta_{ij} (w ) -\theta_i(w) \theta(w) \right
    ) h_i(a_1) h_j(a_2) = \\
  &=\,\frac{1}{2} f_{ij} (0) h_i(a_1) h_j(a_2),
  \\
  & h_i(a_1)dz_1 \,=\, p_1^*\omega_i (a_1), \qquad h_j(a_2)dz_2 \,=\,
    p_2^*\omega_j (a_2) .
\end{align*}
Finally we get that
\begin{gather}
  \label{B12:final}
  B_{12}(a) dz_1\wedge dz_2 \,\,=\,\, \theta[\zeta]^2 (\tphi(a)) \cdot
  \Big [ f(0) \cdot \Omega + \frac{1}{2} f_{ij} (0) p_1^*\omega_i
  \wedge p_2^* \omega_j (a) \Big ].
\end{gather}
  
Next we perform the computation of $A_{12}(a)$ (it is much simpler).
Recall \eqref{eq:againtrisecant:2}:
\begin{gather*}
  A(z)\,:=\, - \one (z) \cdot \one (a_2, \,a_1) \cdot \due (z_1,\, a_1
  ) \cdot \due (a_2,\, z_2 ).
\end{gather*}
Since $\theta[\zeta]$ is odd,
\begin{gather*}
  \due(a_i,\,a_i) \,\,=\,\, \theta[\zeta](0) \,\,=\,\, 0.
\end{gather*}
So among the terms obtained in the process of computing $A_{12}(a)$
only those in which the last two factors are both differentiated can
be nonzero. This means that
\begin{gather*}
  A_{12}(a) \,\,=\,\, -\one(a) \one(a_2,\,a_1) \due _1(a_1,\,a_1)
  \due_2(a_2,\,a_2).
\end{gather*}
By \eqref{eq:omega:zeta},
\begin{align*}
  &    \due_1(a_1,\,a_1)  dz_1 \,=\, \theta[\zeta]_i(0) h_i ( a_1) dz_1 \,=\, p_1^*\omega_\zeta(a), \\
  &    \due_2(a_2,\,a_2) dz_2\,=\, -\theta[\zeta]_i(0) h_i ( a_2) dz_2 \,=\, -
    p_2^*\omega_\zeta (a).
\end{align*}
Moreover,
\begin{gather}
  \notag \begin{aligned}
    & \one(a)\,=\, \theta(w + \tphi(a)) , \\
    & \one(a_2,\,a_1) \,=\, \theta( w - \tphi(a)),
  \end{aligned}\qquad
  \one(a)\cdot \one(a_2,\,a_1) \,=\, s_w ( \tphi(a)),\\
  \label{A12:final}
  A_{12}(a)dz_1 \wedge dz_2\, =\, s_w(\tphi(a)) \cdot
  p_1^*\omega_\zeta \wedge p_2^*\omega_\zeta.
\end{gather}
Plugging \eqref{A12:final} and \eqref{B12:final} into \eqref {eq:10}
completes the proof.
\end{proof}

\begin{proposition}
  \label{prop:operativepullback} There is a unique isomorphism
  \begin{gather*}
    \isom\,\,: \,\,\phi^* \mathbf L \,\,\stackrel{\cong}{\lra}\,\,
    K_S(2\Delta)
  \end{gather*}
  such that
  \begin{gather*}
    \isom \left ( \phi^* (s_\zeta ) \right ) \,\,=\,\,
    \frac{1}{c(\zeta)}(p_1^*\omega_\zeta) \wedge (p_2^* \omega_\zeta)
  \end{gather*}
  for all $\zeta\,\in\,\Theta_{reg}$, where
  $c(\zeta)\,\in\,\mathbb{C}^\ast$ is the constant defined in Lemma
  \ref{allthetalemma}.

  Furthermore, for any $f\,\in\, H^0(J^0(X),\,\mathbf{L})$,
  \begin{equation}
    \label{eq:operativepullback}
    \Psi(\phi^\ast f)\,\,=\,\,f(0)\Omega+\frac{1}{2}\sum_{i,j=1}^g
f_{ij}
    (0) \, p_1^\ast
    \omega_i\wedge p_2^\ast \omega_j.
  \end{equation}
\end{proposition}

\begin{proof}
By proposition \ref{prop:omega:zeta}, both $ \phi^* \mathbf L$ and
$K_S(2\Delta)$ are line bundles associated with the same divisor. So
they are isomorphic.  Let
\begin{gather*}
\psi\,\,:\,\, \phi^* \mathbf L
\,\,\stackrel{\cong}{\longrightarrow}\,\, K_S(2\Delta).
\end{gather*}
be a fixed isomorphism.  It follows from Proposition
\ref{prop:omega:zeta}(b) that $\psi (\phi^*_\zeta)$ and
$p_1^* \omega_\zeta\wedge p_2^* \omega_\zeta$ are sections of
$K_S(2\Delta)$ with the same divisor. So there is
a $\lambda_\zeta \,\in\, \mathbb C^*$ such that
\begin{gather*}
\psi (\phi^*s_\zeta)\,\, =\,\, \lambda_\zeta \cdot p_1^*
\omega_\zeta\wedge p_2^* \omega_\zeta.
\end{gather*}
Let $\isom_\zeta\,:=\,\frac{1}{\lambda_\zeta c(\zeta)} \cdot \psi$
for each $\zeta\,\in\,\Theta_{reg}$. By construction,
$$
\isom_\zeta \left ( \phi^* (s_\zeta ) \right ) \,\,=\,\,
\frac{1}{c(\zeta)}(p_1^*\omega_\zeta) \wedge (p_2^* \omega_\zeta).
$$
Fix an odd regular characteristic $\zeta_0$. Using \eqref{szeta} we
have
\begin{gather*}
    \isom_{\zeta_0}(\phi^\ast
    \theta[\zeta_0]^2)\,\,=\,\,\isom_{\zeta_0}(
    c(\zeta_0) \phi^\ast
s_{\zeta_0})\,\,=\,\,p_1^\ast
    \omega_{\zeta_0}\wedge p_2^\ast \omega_{\zeta_0}.
  \end{gather*}
  By Proposition \ref {prop:pullback}, for any
  $f\,\in\, H^0(\mathbf{L})$ we have
  \begin{gather*}
    \isom_{\zeta_0}(\phi^\ast f)\,\,=\,\,\frac{f\circ
      \tphi}{\theta[\zeta_0]^2\circ \tphi}\isom_{\zeta_0}(\phi^\ast
    \theta[\zeta_0]^2)\,\,=\,\,\frac{f\circ
      \tphi}{\theta[\zeta_0]^2\circ \tphi}p_1^\ast
    \omega_{\zeta_0}\wedge p_2^\ast
    \omega_{\zeta_0}\\
    =\, f(0)\Omega+\frac{1}{2}\sum_{i,j=1}^g
    f_{ij} (0)\, p_1^\ast \omega_i\wedge p_2^\ast
    \omega_j.
  \end{gather*}
  To prove the proposition it suffices to show that
  $\Psi_\zeta\,=\,\Psi_{\zeta_0}$ for every
  $\zeta \,\in\, \Theta_{\rm reg}$.  Since all
  $\{\isom_\zeta\}_{\zeta\in\Theta_{reg}}$ are mutually proportional,
  in order to prove that $\Psi_\zeta\,=\,\Psi_{\zeta_0}$ for all
  $\zeta \,\in\, \Theta_{\rm reg}$ it is enough to show that
  $\Psi_{\zeta_0}(\phi^\ast s_\zeta)\,=\,\Psi_\zeta(\phi^\ast
  s_\zeta)$.

This is rather straightforward: Indeed,
$$
\isom_{\zeta_0}(\phi^\ast
s_\zeta)\,=\,s_\zeta(0)\Omega+\frac{1}{2}\sum_{i,j=1}^g
(s_\zeta)_{ij}(0)\, p_1^\ast
\omega_i\wedge p_2^\ast \omega_j.
$$
Using Proposition \ref{prop:omega:zeta}(a) and Lemma
\ref{allthetalemma}, this is equal to
$$\frac{1}{c(\zeta)}\sum_{i,j=1}^g \frac{\partial
\theta[\zeta]}{\partial z_i}(0)\frac{\partial
\theta[\zeta]}{\partial z_j}(0)p_1^\ast \omega_i\wedge
p_2^\ast \omega_j\,=\,\frac{1}{c(\zeta)}p_1^\ast\omega_\zeta\wedge
p_2^\ast \omega_\zeta=\Psi_\zeta(\phi^\ast s_\zeta).
$$
This completes the proof.
\end{proof}

\begin{remark}
  The right hand side of \eqref{eq:operativepullback} vanishes
  precisely when $f\,\in\,\mathbb{V}_{00}$. Note that this is a proof
  of \eqref{eq:2}.
\end{remark}

We give one last expression for the pullback map $\phi^\ast$.\\
In our notation, the aforementioned Theorem 6 of \cite[part D]{Gu3}
says the following:
\begin{equation}\label{eq:trisecant}
\frac{f\circ \phi}{E^2}p_1^\ast dz\wedge p_2^\ast dz
\ =\ f(0)\Omega +\frac{1}{2} \sum_{i,j}
  f_{ij}(0) \, p^\ast \omega_i\wedge q^\ast \omega_j,
\end{equation}
where $z\,:\,\widetilde{X}\,\longrightarrow\, \mathbb{C}$ is the
coordinate given by the uniformization theorem, and
$E\,:\,\widetilde{X}\times \widetilde{X}\,\longrightarrow\,
\mathbb{C}$ is the classical Prime Function (see \cite[part A]{Gu3}).

The left-hand side of \eqref{eq:trisecant} is a new expression for
$\phi^\ast f$, and makes the intrinsic nature of the pullback more
evident. In fact the Prime Function $E$ is intrinsic.

\section{Comparison of $\sigma_X$ and $\eta_X$}
\label {section:6}

We now proceed to compare the section $\sigma_X$ constructed in
Proposition \ref{prop2} with another section denoted $\eta_X$ that was
constructed in \cite{EPG}, building on work in \cite{cpt}. This section
$\eta_X$ is constructed using basic Hodge theory, and it describes the
second fundamental form of the Torelli map with respect to the Siegel
metric. More precisely, the Torelli map
$j\,:\,M_g\,\longrightarrow\, A_g$ is an embedding outside the
hyperelliptic locus. When $g\, \geq\, 3$, the second fundamental form of
$j(M_g)\, \subset\, A_g$ with respect to the Siegel metric on $A_g$ at a
non-hyperelliptic point $[X]\,\in\, M_g$ coincides is a map
\begin{equation}
  I_2(K_X)\,\, \longrightarrow\,\, \text{Sym}^2(H^0(X,\, 2K_{X})).
\end{equation}
where
$I_2(K_X)\,\subset \, H^0(X,\, K_X)^{\otimes 2}$ is the kernel of the natural map
$$
m\, :\, \text{Sym}^2(H^0(X,\, K_X)) \,\longrightarrow\, H^0(X,\, 2K_X).
$$
One can identify $I_2(K_X)$ with a subspace of $H^0(S, K_S(-2\Delta))$ and $\text{Sym}^2(H^0(X,\, 2K_{X}))$ with  a subspace of $H^0(S, 2K_S)$. With these identification the second fundamental form coincides with the multiplication map
$$
 Q \,\longmapsto\,
 Q\cdot \eta_X;
 $$
see \cite{EPG}.

The form $\eta_X$ has the following defining property: Fix any holomorphic
coordinate chart $(U,\,z)$ on $X$. Since any holomorphic $1$-form on a
Riemann surface is closed, there is a natural inclusion map
$$j_z\,:\, H^0(X,\,K_X(2z(0)))\,\hookrightarrow\,
H^1(X-\{z(0)\},\, \mathbb{C})\,\cong\,H^1(X,\,\mathbb{C}).$$ We have
$h^0(X,\,K_X(2z(0)))\,=\,g+1$ by Riemann-Roch and hence
$\dim j_z^{-1}(H^{0,1}(X,\mathbb{C}))\,=\,1$. Then there is a unique
$\eta_z\,\in\, H^0(X,\,K_X(2z(0)))$ such that
\begin{equation*}
  \langle \eta_z\rangle_{\mathbb{C}}\,\,=\,\,j_z^{-1}(H^{0,1}(X,\,\mathbb{C}))\quad\,\, \text{ and }
  \quad \eta_z\,=\,\left(\frac{1}{z^2}+h(z)\right)dz
\end{equation*}
on $U$ with $h$ holomorphic. This construction defines a collection of
meromorphic forms $\eta_z$ for all possible coordinate charts
$(U,\,z)$. Then one has
\begin{equation}
  \label{eq:bigetaproperty}
  \eta_X|_{X\times \{z(0)\}}\,\,=\,\,\eta_z
\end{equation}
for all charts $(U,\,z)$; here $\eta_X|_{X\times \{z(0)\}}$ is the
contraction by $\frac{\partial }{\partial z}$.

\begin{remark}
  The form $\eta_X$ can be characterized as follows: It is the only
  element of $H^0(S,\, K_{S}(2\Delta))$ with cohomology class of pure
  type $(1,\,1)$ in $H^2(S- \Delta,\,\mathbb{C})$ and with biresidue
  $1$ on the diagonal.  See \cite[Theorem 5.4]{bcfp}. See also
  \cite{Lo}.
\end{remark}

\begin{remark}
  The form $\eta_X$ appears in the unpublished book of Gunning
  \cite{Gu2}. In \cite{Gu2}, the form $\eta_X$ is referred to as the
  \emph{Intrinsic Bidifferential}, and it is constructed by imposing
  \eqref{eq:bigetaproperty}. We briefly remark that $\eta_X$ is indeed
  intrinsic (or canonical), whereas the Canonical Bidifferential
  $\Omega_X$ is not, as its construction depends on the choice of a
  marking on $X$.
\end{remark}

\begin{proposition}
  Fix a marking $\{a_i,\,b_i\}_{i=1}^g$ on $X$, and let
  $\omega_i\,\in\, H^0(X,\, K_X)$ be normalized as in \eqref{nf}. Let
  $\tau_{ij}=\int_{b_i}\omega_j$ be the period matrix of $X$ with
  respect to this marking. Let $\Omega_X$ be the Canonical
  Bidifferential of Proposition \ref{canonical:bidifferential} with
  respect to this marking. Then
  \begin{equation}
    \label{eq:etaexplicit}
    \eta_X\,\,\,=\,\,\,\Omega_X-\pi\sum_{i,j}
(\Im \tau)^{ij} p^\ast \omega_i\wedge q^\ast \omega_j,
  \end{equation}
  where $\Im \tau$ is the imaginary part of $\tau$ and
  $(\Im \tau)^{ik}\Im \tau_{kj} = \delta_{ij}$.
\end{proposition}

\begin{proof}
  The above equality can be found in the proof of Theorem 4.23 of
  \cite{Gu2}.  In that proof, $\eta_X$ (which is called $\mu_X$ there)
  is used in constructing $\Omega_X$ (which is called
  $\widehat{\mu}_X$ there).
\end{proof}

All that is needed now is to find explicitly a generator of the line
$\mathbb{S}$ in \eqref{e8}, for then equation \eqref{eq:pullback} will give an
explicit expression for $\sigma_X$ that can be easily compared with
\eqref{eq:etaexplicit}.

This will be achieved through the following two lemmas.

\begin{lemma}\label{linearalgebra}
  Let $\{f_u\}$ be a unitary base of $\mathbb{V}$ in \eqref{e7}. Then
  $\mathbb{S}=V_0^\perp$ is generated by the section
$$s\,\,:=\,\,\sum_{u} \overline{f_u(0)}f_u.$$
\end{lemma}

\begin{proof}
  This is actually a simple exercise in linear algebra. Since the $f_u$ are
  linearly independent and $\Thetatwo$ is base point free, it is
  immediate that $s\,\neq\, 0$. Therefore, it suffices to show that
  $s\perp V_0$.  For this, notice that if
  $f\,=\,\sum_u c_u f_u\,\in\, V_0$, then
  $$\langle s,\, f\rangle\,\,=\,\,\sum_u \overline{f_u(0)c_u}\,\,=\,\,\overline{f(0)}\,\,=\,\,0.$$
  This completes the proof.
\end{proof}

\begin{lemma}
  \label{igusa80}
  The $2^g$ sections
  $$\left\{\theta_u(z)= \sum_{\xi \in \mathbb{Z}^g} e^{2\pi \sqrt{-1}(\xi+u)^t\tau(\xi+u)+4\pi \sqrt{-1}(\xi+u)^tz}\right\}_{u\in \frac{2^{-1}\mathbb{Z}^g}{\mathbb{Z}^g}}$$
  have the same norm and form an orthogonal basis of $\mathbb{V}$,
  that is
    $$\langle \theta_u,\,\theta_v\rangle\,\,=\,\,C\delta_{uv}$$
    for some constant $C\,>\,0$.
  \end{lemma}

  \begin{proof}
    See for example \cite[p.~80]{Igusa}.
  \end{proof}

\begin{remark}
  The constant $C$ in Lemma \ref{igusa80} can be computed in terms of
  the period matrix $\tau$. Since the precise value of this constant
  plays no role in our computation, it will be kept as it is.
\end{remark}

\begin{proposition}
  The following holds:
  \begin{equation}\label{eq:pullbackS}
    \sigma_X\ =\ \Omega+\frac{1}{2\sum_{u\in \frac{2^{-1}\mathbb{Z}^g}{\mathbb{Z}^g}} |\theta_u(0)|^2}
    \sum_{i,j=1}^g\sum_{u\in \frac{2^{-1}\mathbb{Z}^g}{\mathbb{Z}^g}}
    \overline{\theta_u(0)}\frac{\partial^2 \theta_u}{\partial z_i \partial z_j}(0)
    p_1^\ast \omega_i\wedge p_2^\ast \omega_j.
  \end{equation}
\end{proposition}

\begin{proof}
  Let $s$ be a generator of $\mathbb{S}$. Then $\sigma_X$ is equal to
  the pullback $\phi^\ast s$ normalized to be $1$ when restricted to
  the diagonal $\Delta\, \subset\, X\times X$. Using equation \eqref{eq:operativepullback} we have
$$\isom(\phi^\ast s)\,\,\,=\,\,\,
s(0)\Omega+\frac{1}{2}\sum_{i,j=1}^g \frac{\partial^2 s}{\partial
  z_i\partial z_j}(0)p_1^\ast \omega_i \wedge p_2^\ast \omega_j.$$
Since $\Omega$ is equal to $1$ when restricted to the diagonal, normalizing, we have
\begin{equation}\label{eq:esseX}
  \sigma_X\,\,=\,\,\Omega+\frac{1}{2s(0)}\sum_{i,j=1}^g \frac{\partial^2 s}{\partial z_i\partial z_j}(0)
  p_1^\ast \omega_i\wedge p_2^\ast \omega_j.
\end{equation}
In view of Lemma \ref{igusa80}, the functions
$
\theta_u/ {\sqrt{C}}$ constitute a unitary basis of $\mathbb{V}$. So by Lemma \ref{linearalgebra} 
the line $\mathbb{S}$ is generated by
\begin{equation}\label{eq:esse}
  s\,\,=\,\, \frac{1}{C}\sum_{u\in \frac{2^{-1}\mathbb{Z}^g}{\mathbb{Z}^g}}\overline{\theta_u(0)}\theta_u.
\end{equation}
Plugging \eqref{eq:esse} into
\eqref{eq:esseX} we conclude that \eqref{eq:pullbackS} holds.
\end{proof}

We now restrict to the special case of $g\,=\,1$.

\begin{theorem}\label{final}
The forms $\sigma_X$ and $\eta_X$ differ for any $[X]$ in a nonempty
open subset, in the complex topology, of the moduli space $M_{1,1}$.
\end{theorem}

\begin{proof}
Set $\mathbb H\,\,:=\,\, \{\tau\,\in\, \mathbb C \,\,\big\vert\,\, {\rm
Im}\, \tau \,>\,0\}$. For $\tau \,\in \,\mathbb H$ consider the elliptic
curve $X_\tau\,=\,{\mathbb{C}}/\, ({\mathbb{Z}+\tau\mathbb{Z}})$. Denote
  by $\Omega_\tau$,\, $\eta_\tau$ and $s_\tau$ the forms
  $\Omega_{X_\tau}$,\, $\eta_{X_\tau}$ and $s_{X_\tau}$ respectively.
for $X_\tau$. For $g\,=\,1,$ stressing the dependence on $\tau$, the equations
  \eqref{eq:pullbackS} and \eqref{eq:etaexplicit} become
  \begin{gather*}
    \sigma_\tau\,\,=\,\,\Omega_\tau+\frac{1}{2 w(\tau)} \sum_{u\in \{0, 1/2\}}
    \overline{\theta_u(0;\,\tau)}\, \frac{\partial^2
      \theta_u}{\partial z^2}(0;\tau) p_1^\ast \omega_1\wedge p_2^\ast
    \omega_1
    ,\\
    \eta_\tau\,\,=\,\,\Omega_\tau-\frac{\pi }{{\rm Im}\, \tau} p^\ast
    \omega_1\wedge q^\ast \omega_1,
  \end{gather*}
where
\begin{gather*}
w(\tau)\,\,:=\,\, |\theta_0(0;\,\tau)|^2 + |\theta_{1/2}(0;\,\tau)|^2.
\end{gather*}
By the heat equation for second order theta functions (see \cite{vg}) we
have the following:
  \begin{gather*}
    8\pi \sqrt{-1} \frac{\partial \theta_u}{\partial \tau} \,\,=\,\,
    \frac{\partial^2 \theta_u}{\partial z^2}.
  \end{gather*}
  Hence using the fact that $\theta_u(0;\,\tau) $ is holomorphic in
  $\tau$,
  \begin{gather*}
    \sigma_\tau\,\,=\,\,\Omega_\tau+\frac{4\pi\sqrt{-1}}{ w(\tau)}
    \frac{\partial w} { \partial \tau} (\tau) p_1^\ast \omega_1\wedge
    p_2^\ast \omega_1.
  \end{gather*}
  Therefore, $\sigma_\tau \,=\, \eta_\tau$ if and only if
  \begin{gather*}
    \frac{4\pi\sqrt{-1} }{ w(\tau)} \frac{\partial w} { \partial \tau}
    (\tau) \,\, =\,\, -\frac{\pi}{{\rm Im}\, \tau} .
  \end{gather*}
Setting $\tau \,=\, x + \sqrt{-1} y$, and keeping in mind that $w$
is a real function, we conclude that the following three statements are equivalent:
\begin{enumerate}
\item $\sigma_\tau \,=\, \eta_\tau$,

\item $4\sqrt{-1} y w_\tau + w\,=\, 0$, and

\item the derivatives satisfy
\begin{equation}\label{system}
w_x \,=\,0 \,=\,  2y w_y + w.
\end{equation}
\end{enumerate}
  We are going to express $w$ in a simpler form.  From
Lemma   \ref{igusa80}, for $g\,=\,1$ and $z\,=\,0$,
  \begin{gather}
  \notag  \theta_u(0;\,\tau)\,=\, \sum_{\xi \in \mathbb{Z}} e^{2\pi
      \sqrt{-1}(\xi+u)^2\tau} \qquad
    u\,\in\, \{0,\, 1/2\} ,\\
    \label{series} | \theta_u(0;\,\tau) |^2 \,=\, \sum_{\xi, \zeta \in
      \mathbb{Z}} e^{2\pi \sqrt{-1}(\xi+u)^2\tau -2\pi \sqrt{-1}
      (\zeta + u )^2 \overline{\tau} }.
  \end{gather}
  For $u\,=\,0$, we have
  $\xi^2\tau - \zeta^2 \overline{\tau}\,=\, x( \xi^2 -\zeta^2) +
  \sqrt{-1}y ( \xi^2 +\zeta^2)$.  Set
  \begin{gather*}
    \Gamma \,\,:=\{(m,\,n)\,\in\, {\mathbb Z}^2\,\, \big\vert\,\, m+ n
    \,\in\, 2\mathbb Z\}.
  \end{gather*}
This $\Gamma$ is a subgroup of $\mathbb Z^2$ of index 2.  We use the change
  of variables $ m\,=\, \xi - \zeta$, $n\, =\, \xi + \zeta$ to get an
isomorphism between ${\mathbb Z}^2 $ and $ \Gamma$. With this change
  of variables, the series in \eqref{series} for $u\,=\,0$ becomes
  \begin{gather*}
    |\theta_0(0;\,\tau)|^2 \,\,=\,\, \sum _{(m,n)\in
      \Gamma}e^{2\pi\sqrt{-1} x mn -\pi y (m^2 +n^2)}.
  \end{gather*}
  For $u\,=\,1/2$ there is a similar computation:
  \begin{gather*}
    (\xi+ 1/2)^2\tau - (\zeta+1/2)^2
    \overline{\tau}\,\,=\, \,x ( \xi^2 -\xi +\zeta^2 -\zeta )
    +\sqrt{-1} y ( \xi^2 +\zeta^2 + \xi +\zeta +1/2).
  \end{gather*}
  Substituting $m\,=\, \xi - \zeta$ and $n \,=\, \xi + \zeta +1 $ we
  get the following:
  \begin{gather*}
    |\theta_{1/2}(0;\,\tau)|^2 \,= \,\sum _{(m,n)\in \mathbb Z^2 -
      \Gamma}e^{2\pi\sqrt{-1} xmn -\pi y (m^2 +n^2)},\\
    w(x+\sqrt{-1}y)\,=\, \sum _{(m,n)\in \mathbb Z^2}e^{2\pi\sqrt{-1}
      x mn -\pi y
      (m^2 +n^2)},\\
    w_x(x+\sqrt{-1}) \,=\, \sum _{(m,n)\in \mathbb Z^2} (2\pi
    \sqrt{-1} m n)\cdot e^{2\pi\sqrt{-1} x mn -\pi
      (m^2 +n^2)},\\
    w_{xx}(\sqrt{-1}) \,=\, - \sum _{(m,n)\in \mathbb Z^2} (2\pi m
    n)^2\cdot e^{-\pi (m^2 +n^2)} \,<\, 0.
  \end{gather*}
  This shows that the equation $w_x\,=\,0$ is not satisfied in
  general.
\end{proof}

\begin{remark}
  One can prove that $\sigma_\tau \,=\, \eta_\tau$ for
  $\tau\,=\,\sqrt{-1}$.  Indeed, both forms belong to
  $H^0(S,\, K_S(2\Delta)) $ which is 2-dimensional, and both are
  invariant under $\operatorname{Aut}(X_\tau,+)$. Since this group acts nontrivially on
  $H^0(X_\tau,\, K_{X_\tau})$, we conclude that it also acts nontrivially on
  $H^0(S,\, K_S)$. So
  $\dim H^0(S,\, K_S(2\Delta))^{\operatorname{Aut}(X_\tau)} = 1
  $. Therefore $\sigma_\tau$ and $\eta_\tau$ are proportional.  By the
  normalization of the biresidue they coincide.

The same happens for $\tau\, =\, e^{\frac{\pi}{3} \sqrt{-1}}$. Actually, this
is not surprising. For these elliptic curves the holomorphic 1-form
is not invariant with respect to the automorphisms.  Similarly,
$H^0(X_\tau, 2 K_{X_\tau} )^{\operatorname{Aut}(X_\tau)} \,=\,
0$. Given two projective structures, that are invariant under
$\operatorname{Aut}(X_\tau)$, their difference is an invariant
quadratic differential. This differential vanishes, so the two
projective structures must coincide.

By a computation similar to the one above for $w_{xx}$, one can show
that $\tau \,=\, \sqrt{-1}$ is an isolated point of the locus
  $\{\tau \,\in\, \mathbb H\,\,\big\vert\,\, \sigma_\tau \,=\,
  \eta_\tau\}.$ It would be interesting to understand better this
  locus.  Since $w$ is real analytic, \eqref{system} is a system of
  real analytic equations, so this locus is real analytic in
  $\mathbb H$.
\end{remark}

In conclusion, we compare the projective structures on $X$ induced by
$\eta_X$ and $\sigma_X$.

\begin{definition}
  Let $X$ be a compact Riemann surface. A projective structure on $X$
  subordinate to the complex structure of $X$ is a collection of
  holomorphic coordinate charts $\{(U_i,\,z_i)\}_{i\in I}$ such that
  $\bigcup_{i\in I} U_i\,=\,X$ and all the transition functions
  $z_i\circ z_j^{-1}$ are restrictions of M\"obius transformations on
  each connected component of the domain of $z_i\circ z_j^{-1}$, that
  is
$$z_i\,=\,\frac{a_{ij}z_j+b_{ij}}{c_{ij}z_j+d_{ij}}$$
for constants $a_{ij},\, b_{ij},\, c_{ij},\, d_{ij}$.
\end{definition}

We briefly recall how nonzero global sections of $K_S(2\Delta)$ induce
projective structures on $X$.  Consider the structure sequence of
$3\Delta$ tensored by $K_S(2\Delta)$:
\begin{equation}
  \label{eq:3Delta}
  0\,\longrightarrow\, K_S(-\Delta) \,\longrightarrow\, K_S(2\Delta)\,\xrightarrow{\,\,\,\rho\,\,\,}\,
  K_S(2\Delta)|_{3\Delta}\,\longrightarrow\, 0.
\end{equation}
Take the long exact sequence of cohomologies associated to
\eqref{eq:3Delta}:
\begin{equation}
  \label{eq:restrictionsequence}
  0 \,\longrightarrow\, H^0(S,\, K_S(-\Delta)) \,\longrightarrow\,
  H^0(S,\, K_S(2\Delta))\,\xrightarrow{\,\,\,\rho\,\,\,}\, H^0(K_S(2\Delta)|_{3\Delta})\,\longrightarrow\,  \cdots
\end{equation}
There is a canonical section of $K_S(2\Delta)|_{2\Delta}$, and the
space $\wp(X)$ of all projective structures on $X$ subordinate to the
complex structure of $X$ is in a bijective correspondence with the set
of sections of $H^0(K_S(2\Delta)| _{3\Delta})$ whose restriction to
$2\Delta$ coincides with the canonical section.  This construction is
explained in \cite{BR1,bcfp}. See also \cite{Gu1} and \cite {Tyu} 
for more background on projective structures.

Both $\eta_X$ and $\sigma_X$ have biresidue $1$ along the
diagonal. Coupled with \eqref{ka4} this means that they both restrict
to the canonical section along $2\Delta$. Therefore they induce two
projective structures on $X$ that we denote by $\beta^\eta_X$ and
$\beta^\theta_X$ respectively.

\begin{proposition}
  \label{g1} If $X$ is a compact Riemann surface of genus $1$, then
  $H^0(S,\, K_S(-\Delta)\,=\,0$.  Hence the restriction map
  $\rho\,:\,H^0(K_S(2\Delta)) \,\longrightarrow\,
  H^0(K_S(2\Delta)_{|3\Delta})$ is injective.
\end{proposition}

\begin{proof}
  The space $H^0(X,\, K_X)$ is generated by the nowhere zero section
  $\omega_1$, and $H^0(S,\, K_S)$ is generated by the section
  $p_1^\ast \omega_1\wedge p_2^\ast \omega_1$.  Since this section is
  nowhere zero, the only section of $H^0(S,\, K_S)$ that vanishes on the
  diagonal is $0$, and therefore we have $H^0(S,\, K_S(-\Delta))\,=\,0$. From the
  long exact sequence in \eqref{eq:restrictionsequence} we see that the
  kernel of $\rho$ is $H^0(S,\, K_S(-\Delta))$. Therefore, $\rho$ is
  injective.
\end{proof}

\begin{corollary}\label{cor-d}
  For $g\,=\,1$,
$$\beta^\theta_X\,\,\neq\,\, \beta^\eta_X$$ for all $X$ outside a
proper real analytic subvariety of $M_{1,1}$.
\end{corollary}

\begin{proof}
The projective structures $\beta^\eta$ and $\beta^\theta$ coincide
if and only if $\rho(\eta_X)\,=\,\rho(\sigma_X)$. For genus $1$, the
map $\rho$ is injective, therefore we have
$\beta^\eta_X\,=\,\beta^\theta_X$ if and only if $
\eta_X\,=\,\sigma_X$. By Theorem \ref{final} this is the case
for all $[X]$ in a proper real analytic subvariety of $M_{1,1}$.
\end{proof}

\section*{Acknowledgements}

We are very grateful to the two referees for their helpful comments to improve the
exposition and also for giving some better proofs and pointing out references.
The second author would like to thank
Elisabetta Colombo, Paola Frediani and Gian Pietro Pirola for
interesting discussions related to the subject of this paper.  In
particular he would like to thank Elisabetta Colombo for the crucial
suggestion that $\eta_X\, \neq\, \sigma_X$. The second and third authors
would like to thank Emanuele Dolera for help with special values of
the Jacobi theta function.  The first author is partially supported by
a J. C. Bose Fellowship (JBR/2023/000003). The second and third
authors were partially supported by INdAM-GNSAGA, by MIUR PRIN 2022:
20228JRCYB, ``Moduli spaces and special varieties'' and by FAR 2016
(Pavia) ``Variet\`a algebriche, calcolo algebrico, grafi orientati e
topologici''.

\end{document}